\documentclass[12pt]{amsart} 
\usepackage[left=3cm,top=2.5cm,bottom=2.5cm,right=3cm]{geometry}
\usepackage{times}
\usepackage{amssymb, amsmath, amsthm}
\usepackage{amsmath}
\usepackage{amssymb,amsfonts,stmaryrd,mathrsfs}
\usepackage{latexsym,amsthm,mathtools,bbm}
\usepackage{mathtools}
\usepackage{graphicx,xspace}
\usepackage{epsfig}
\usepackage{enumitem}
\usepackage[usenames,dvipsnames]{xcolor}
\usepackage{tikz}
\usepackage{gensymb}
\usepackage[T1]{fontenc}
\usepackage[utf8]{inputenc}
\usepackage{bm}
\usepackage{subcaption}

\definecolor{green}{RGB}{0,127,0}
\definecolor{red}{RGB}{191,0,0}
\usepackage[colorlinks,cite color=red,link color=green,pagebackref=true]{hyperref}
\usepackage[capitalize]{cleveref}
\usepackage{todonotes}

\newtheorem{thm}{Theorem}[section]

\newtheorem{lemma}[thm]{Lemma}
\newtheorem{prop}[thm]{Proposition}
\newtheorem{defi}[thm]{Definition}

\newcommand{\ZZ}{\mathbb{Z}}
\newcommand{\QQ}{\mathbb Q}

\newcommand{\Fd}{F^{\pi_d}}

\newcommand{\bfp}{\mathbf{p}}
\newcommand{\bfq}{\mathbf{q}}

\newcommand{\YY}{\mathbb Y}

\newcommand{\dt}{\partial_t}

\title[]{Differential equations for bipartite maps with bounded face degrees}

\author[V.~Bonzom]{Valentin Bonzom}
\address{LIGM, CNRS UMR 8049, Université Gustave Eiffel, Champs-sur-Marne, France}
\email{valentin.bonzom@univ-eiffel.fr}

\thanks{V.~B.~was partially supported by the ANR-23-CE48-0018 CartesEtPlus.
}
\begin{document}
	
\begin{abstract}
	In recent years, integrable hierarchies have been used to great advantage for the enumeration of combinatorial maps. They have led to recurrence formulas with respect to the size and genus of the maps, e.g. for triangulations, bipartite quadrangulations and bipartite maps, and for constellations. These formulas are not only remarkably simple but also provide the fastest way of calculating these numbers of maps. With the exception of Louf's work on constellations, it has however remained a challenge to obtain recurrence formulas that control the degrees of the faces of the maps. Here we show how to achieve this for bipartite maps with bounded face degrees. By combining equations from the KP hierarchy and from the Virasoro constraints, a differentially algebraic system is obtained. It couples the generating functions of bipartite maps with bounded root face degrees while controlling the numbers of edges, black vertices, white vertices and number of faces of each degree (and in particular the genus). Finally, this system of ODEs is shown to give recurrence formulas that allows to calculate all the corresponding numbers of maps.
\end{abstract}
	
\maketitle
	
\section{Introduction}
Combinatorial maps (or simply, maps) are multi-graphs properly embedded in surfaces, without crossings and up to deformations. This gives maps more structure than graphs and a canonical genus which is that of the embedding surface. Maps of genus $g$ in fact provide a graphical representation of ramified coverings of genus $g$ of the 2-sphere \cite{LandoZvonkin2004}, so that enumerating maps also means computing Hurwitz numbers.

Tutte in the 60s pioneered the enumeration of planar maps, i.e. maps of genus 0, but maps of positive genus in general require different methods. A set of equations {\it à la} Tutte, called below the \emph{Virasoro constraints} (a.k.a. loop equations), can be written for the set of generating functions of maps of genus $g$ and $n$ rooted faces \cite{BenderCanfield1986}. These equations can be solved recursively in $2g+n$ using the so-called topological recursion of Eynard-Orantin \cite{EO, Eynard:book}. This approach provides interesting results on the structure of the generating functions at fixed $(g,n)$, showing that they are rational with respect to some generating functions of trees \cite{BenderCanfieldRichmond93}. Such structural results could also be proved bijectively \cite{AlbenqueLepoutre2019}.

However, neither the topological recursion nor the bijective approaches have provided an efficient way of calculating the number of maps of a fixed size (say with respect to number of edges) and fixed genus. It also remained a mystery whether the generating function of maps with respect to their size and genus satisfies a ``nice'' equation, like an ordinary differential equation (ODE) with respect to the size parameter. An exception is the case of maps of genus $g$ with one face for which a nice recurrence exists, due to Harer and Zagier \cite{HarerZagier1986}, and which also has bijective proofs \cite{ChapuyFerayFusy2013}.

\subsection*{The KP hierarchy} A different approach has produced strong results in these directions, based on integrability. This is the approach that we will use here, making use of the KP hierarchy \cite{Kac2013Bombay, JimboMiwa1983}. Let $\mathbf{p} = (p_1, p_2, \dotsc)$ be an infinite set of indeterminates, and $\tau(\mathbf{p})\in R[[p_1, p_2, \dotsc]]$ (the ring of formal power series in the indeterminates $p_1, p_2, \dotsc$ with coefficients in some ring $R$). Following the Japanese school, the KP hierarchy is an infinite system of partial differential equations (PDEs) which are consistent with one another. We say that $\tau(\mathbf{p})$ is a KP tau function if it satisfies the KP hierarchy, whose first equation is the famous KP equation
\begin{equation} \label{KPEq}
	-F_{3,1} + F_{2,2} + \frac{1}{2} F_{1,1}^2 + \frac{1}{12} F_{1,1,1,1} = 0
\end{equation}
with the notation $F\equiv F(\bfp) = \ln \tau(\mathbf{p})$ and $F_{i_1, i_2, \dotsc} \coloneqq \frac{\partial}{\partial p_{i_1}} \frac{\partial}{\partial p_{i_2}} \dotsb F$. Let us give a couple more equations from the hierarchy (obtained from the appendix of \cite{JimboMiwa1983}),
\begin{equation*}
	\begin{aligned}
		&-F_{4,1} + F_{3,2} + F_{2,1}F_{1,1} + \frac{1}{6}F_{2,1,1,1} = 0,\\
		&\begin{multlined}-F_{5,1} + F_{3,3} + \frac{1}{12}F_{3,1,1,1} + \frac{1}{4}F_{2,2,1,1} + \frac{1}{2}F_{3,1}F_{1,1} + \frac{1}{2}F_{2,2}F_{1,1} + F_{2,1}^2 - \frac{1}{24}F_{1,1}F_{1,1,1,1} \\- \frac{1}{12}F_{1,1}^3 - \frac{1}{720} F_{1,1,1,1,1,1} = 0,\end{multlined}\\
		&\begin{multlined}F_{4,2} - F_{3,3} + \frac{1}{24}F_{3,1,1,1} - \frac{1}{8} F_{2,2,1,1} + \frac{1}{4} F_{3,1}F_{1,1} - \frac{1}{4}F_{2,2}F_{1,1} - \frac{1}{2}F_{2,1}^2 + \frac{5}{48}F_{1,1}F_{1,1,1,1} \\+ \frac{5}{24}F_{1,1}^3 + \frac{1}{288} F_{1,1,1,1,1,1} = 0.\end{multlined}
	\end{aligned}
\end{equation*}

In this article we focus on \emph{bipartite maps}, aka {\it dessins d'enfants}, i.e. such that a black (resp. white) vertex is only connected to white (resp. black) vertices. A face of a map is a connected component of the complement of the graph in the embedding surface. A corner of a map is a portion of a face that sits at a vertex between two adjacent edges along the face. In a bipartite map, every face alternates corners incident to black and white vertices and we call the \emph{degree of a face} the number of white (or black) corners of this face. See an example in Figure \ref{fig:Map}.

\begin{figure}
	\includegraphics[scale=.85]{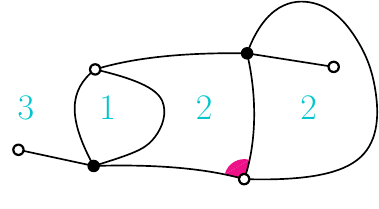}
	\caption{\label{fig:Map}A connected, rooted bipartite map where the face degrees are given. The root is a distinguished corner on a white vertex, here pictured as an angular sector.}
\end{figure}

Let $M_{V_\bullet, V_\circ, \lambda}(n)$ be the number of edge-labeled, bipartite maps, connected or not, with $n$ edges, $V_{\bullet}$ black vertices, $V_{\circ}$ white vertices and faces of degrees $\lambda_1, \lambda_2, \dotsc, \lambda_\ell$, with the notation $\lambda = (\lambda_1, \dotsc, \lambda_\ell)$ ordered non-increasingly. Let $t,u,v$ be three formal parameters. Then let $\tau(t, u, v, \mathbf{p})\in \mathbb{Q}[u,v,\bfp][[t]]$ (the ring of formal power series in $t$ whose coefficients are polynomials in $u,v,p_1, p_2,\dotsc$) be the generating function
\begin{equation} \label{BipTau}
	\tau(t,u,v,\mathbf{p}) = \sum_{n\geq 0} \frac{t^n}{n!} \sum_{V_\bullet, V_\circ, \lambda\vdash n} M_{V_\bullet, V_\circ, \lambda}(n) u^{V_\bullet} v^{V_\circ} p_{\lambda_1} p_{\lambda_2} \dotsm
\end{equation}
It is a KP tau function, as shown by Goulden and Jackson\footnote{This was in fact known before, in mathematical physics, from the connection between maps and matrix integrals. In particular, the origin of integrability for maps can be traced back to the use of orthogonal polynomials to calculate matrix integrals.} \cite{GouldenJackson2008}. The KP hierarchy does not characterize the tau function\footnote{For instance, polynomial solutions form the Sato's Grassmannian \cite{Kac2013Bombay}.}, however the set of Virasoro constraints mentioned earlier does determine $\tau(t,u,v,\bfp)$.
It was already noticed by physicists \cite{KazakovKostovNekrasov1999} that the KP equation \eqref{KPEq} can be used in conjunction with the Virasoro constraints to write a differential equation on $F$, with respect to $t$, in the case of triangulations. This was fully taken advantage of by Goulden and Jackson, who made these techniques available to the combinatorics community and were able to deduce a remarkably simple recurrence relation for triangulations of size $n$ and genus $g$ (triangulations cannot be obtained from \eqref{BipTau}, but the same recipe works).

\subsection*{Bipartite quadrangulations} Carrell and Chapuy \cite{CarrellChapuy2015} took advantage of these techniques to obtain recurrence relations for rooted bipartite quadrangulations (equivalently maps with no constraints on face degrees),
\begin{multline} \label{Quad}
	(n+2)Q_n(u,v,z) = 4uv(u+v)z \delta_{n,2} + 6uv(uv+1)z^2\delta_{n,4} \\+ 4(u+v)z(n-1)Q_{n-2}(u,v,z)
	+ (n^2 - n + 12uv)(n - 2) z^2 Q_{n-4}(u,v,z) \\
	+ 6z^2 \sum_{\substack{i,j\geq 1\\i+j=n-6}} (i+2)(j+2) Q_i(u,v,z) Q_j(u,v,z),
\end{multline}
where $Q_n(u,v,z) = [t^{n+1}] \frac{d\ln\tau}{dt}(t,u,v,\mathbf{p}=(0,z,0, \dotsc))$ is the number of connected, rooted bipartite quadrangulations with $n\geq 2$ edges (setting $Q_n(u,v,z)=0$ for $n<2$). Kazarian and Zograf \cite{KazarianZograf2015} gave a similar recurrence for general bipartite maps (i.e. $p_i=z$ for all $i\geq 1$).

\subsection*{Main result} A natural question is whether such recurrence relations hold for other models of maps, such as quadrangulations, $p$-angulations and more generally maps whose faces can have any degree up to a fixed bound $d$. Here we answer this question positively in the bipartite case. Let $\mathbb{K}_d = \QQ[u,v,\pi_d(\mathbf{p})]$ and $F(t,u,v,\bfp) =\ln\tau(t,u,v,\bfp)\in \mathbb{K}_d[[t]]$ be the generating function of connected bipartite maps. Let $\pi_d$ be the operator that specializes elements of $\mathbb{K}_d[[t]]$ to $p_i=0$ for $i>d$. We denote 
\begin{equation*}
	\Fd_{i_1, i_2, \dotsc}(t,u,v,\pi_d(\bfp)) \coloneqq \pi_d \frac{\partial}{\partial p_{i_1}} \frac{\partial}{\partial p_{i_2}} \dotsm F(t,u,v,\bfp)
\end{equation*}
and in particular $\Fd(t,u,v,\pi_d(\bfp)) = F(t,u,v,(p_1, \dotsc, p_d, 0, \dotsc))$. Expand the following series as
\begin{equation} \label{Fexpansion}
	\Fd (t,u,v,\pi_d(\bfp)) = \sum_{n\geq 1} f_n t^{n}, \qquad \Fd_k(t,u,v,\pi_d(\bfp)) = \sum_{n\geq 1} f^{(k)}_n t^{n},
\end{equation}
with $f_n\equiv f_n(u,v,\pi_d(\bfp))$ and $f^{(k)}_n\equiv f^{(k)}_n(u,v,\pi_d(\bfp))$ elements of $\mathbb{K}_d$, for $k=1, \dotsc, d-1$. Our main result is that one can calculate all $f_n$ via the $f^{(k)}_n$s in a triangular manner.

\begin{thm} \label{thm:Main}
	For all $n\geq 1$, $f_n = f^{(1)}_{n+1}$, and there exists a system of $d-1$ recurrence relations which determines the coefficients $f_{n}^{(k)}$ recursively,
	\begin{equation*}
		\begin{aligned}
			P_1(n) f^{(1)}_n &= R_1(n, f^{(1)}_{<n}, f^{(2)}_{<n}, \dotsc, f^{(d-1)}_{<n})\\
			P_2(n) f^{(2)}_n &= R_2(n, f^{(1)}_{\leq n}, f^{(2)}_{<n}, \dotsc, f^{(d-1)}_{<n})\\
			P_3(n) f^{(3)}_n &= R_3(n, f^{(1)}_{\leq n}, f^{(2)}_{\leq n}, f^{(3)}_{<n}, \dotsc, f^{(d-1)}_{<n})\\
			&\vdots\\
			P_{d-1}(n) f^{(d-1)}_n &= R_{d-1}(n, f^{(1)}_{\leq n}, f^{(2)}_{\leq n} \dotsc, f^{(d-2)}_{\leq n}, f^{(d-1)}_{<n}).
		\end{aligned}
	\end{equation*}
	Here $P_k(n)\in\mathbb{K}_d[n]$ is non-zero, and $R_k$ is polynomial in its variables with coefficients in $\mathbb{K}_d$. Moreover we use the notation $f^{(k)}_{<n} \equiv \{f^{(k)}_{n'}\}_{n'<n}$ and similarly for $f^{(k)}_{\leq n} = f^{(k)}_{<n+1}$.
	
	This holds for all $n\geq 1$ except for $(k,n)=(1,1)$ where the initial condition $f^{(1)}_1=uv$ holds instead.
\end{thm}

Euler's formula gives access to the genus $g$ of such a connected, bipartite map: $2-2g=F-n+V_\bullet+V_\circ$, where $F$ is the total number of faces. As a consequence, knowing the size $n$, together with having the parameters $u, v, p_1, \dotsc, p_d$, allows to extract the number $f_{n,g}(u,v,\pi_d(\bfp))$ of connected, bipartite maps with $n$ edges and of genus $g$,
\begin{equation*}
	f_{n,g}(u,v,\pi_d(\bfp)) = [x^{2-2g+n}]f_n(xu,xv,xp_1, \dotsc, xp_d).
\end{equation*}

In the case $d=2$, Theorem \ref{thm:Main} gives a single equation which is a straightforward generalization of \eqref{Quad} with $z=p_2$ and non-zero $p_1$. We give more examples in Section \ref{sec:ODEs}.

The Virasoro constraints allow to express any $\Fd_{i_1, i_2, \dotsc}$ as polynomials in $\Fd_{j_1, j_2, \dotsc}$ where $j_1, j_2, \dotsb\leq d-1$, see Proposition \ref{prop:VirasoroReduction}. This still leaves many unknowns and it is {\it a priori} unclear if and how to use the KP hierarchy to close the system. It turns out that there is a formulation of the hierarchy that is well-adapted to our needs, due to Dubrovin and Natanzon \cite{DubrovinNatanzon1989}. It expresses any derivatives of $F(\bfp)$ polynomially on the set $\{F_{i, 1^{l}}\}_{i,l}$. By a finite number of rounds of substitution between the Virasoro constraints and the KP hierarchy, one finds an algebraically differential system on $\Fd_1, \dotsc, \Fd_{d-1}$ with respect to $t$, which in turn gives the recurrence of Theorem \ref{thm:Main}.

The fact that the Virasoro constraints, treated as algebraic relations between the quantities $\Fd_{i_1, i_2, \dotsc}$, leave some of those unknown is the origin of the difficulties when solving these equations in various contexts. For example, when restricting attention to planar maps, the Virasoro constraints simplify a bit but one ends up with an equation that features the unknown $\Fd_1, \dotsc, \Fd_{d-1}$ (or rather their planar parts). This gives rise to the notion of equations with catalytic variables \cite{BousquetJehanne2006}. In higher genus, Eynard-Orantin topological recursion manages to bypass the issues of determining these unknown series using complex analysis and well-chosen residues. In our case, we deal with them using the KP hierarchy instead.

\subsection*{Comparison with \cite{Louf2019}} Let us contrast our result with that of Louf \cite{Louf2019}. There, the author uses a different integrable hierarchy than KP, called the Toda hierarchy, which is satisfied by the generating function of bipartite maps with a weight $p_i$ on faces of degree $i$, $q_i$ on white vertices of degree $i$ and $v$ on black vertices. It gives rise to a beautiful functional recurrence for $f_n(u,v,\mathbf{p})$ the number of bipartite maps with a weight $u$ per white vertex, $v$ per black vertex and $p_i$ per face of degree $i$,
	\begin{multline} \label{Louf}
		n(n+1) f_n(u,v,\bfp) = \sum_{k+l=n+1} (uv\delta_{k,1}+kf_{k-1}(u,v,\bfp)) \Bigl(- 2f_l(u,v,\bfp) \\+ (1+u)^{l+1}(1+v)^{l+1}f_l\Bigl(\frac{u}{1+u},\frac{v}{1+v},\bfp\Bigr) + (1-u)^{l+1}(1-v)^{l+1}f_l\Bigl(\frac{u}{1-u},\frac{v}{1-v},\bfp\Bigr) \Bigr)
\end{multline}
It is obviously much more compact and explicit than the system we obtain in Theorem \ref{thm:Main}. It is also more general, as it holds without a fixed bound on the face degrees, and also generalizes to constellations (of which bipartite maps are a special case). However there is a key difference that justifies our work: Equation \eqref{Louf} is \emph{not} an ordinary recurrence on $f_n(u,v,\pi_d(\mathbf{p}))$ (even when restricted to bounded face degrees), because it involves $f_n(u,v,\bfp)$ at transformed values of the variables $u\mapsto u/(1\pm u)$ and $v\mapsto v/(1\pm v)$. This is because \eqref{Louf} is not derived from an ODE on $F(t,u,v,\bfp) \coloneqq \ln\tau(t,u,v,\mathbf{p})$, but instead from a functional equation that involves the transformations $F(t,u,v,\bfp) \mapsto F(t(1\pm u)(1\pm v), \frac{u}{1\pm u}, \frac{v}{1\pm v},\bfp)$, due to the nature of the Toda hierarchy.

Our result therefore fills an academic gap that answers the question whether there exists a recurrence relation, or at least a system of recurrence relations that determines $f_n(u,v,\pi_d(\mathbf{p}))$, or in other words whether it satisfies a \emph{differentially algebraic system of equations with respect to the size variable $t$ only that determines all coefficients}.

\subsection*{Plan} In Section \ref{sec:KP}, we present the equations of KP hierarchy due to Dubrovin and Natanzon \cite{DubrovinNatanzon1989}. In Section \ref{sec:Maps}, we focus on the Virasoro constraints and explain our strategy that combines them together with the KP equations from Section \ref{sec:KP}. In Section \ref{sec:ODEs}, we derive the differentially algebraic system mentioned above on the series $\Fd_1(t,u,v,\pi_d(\bfp))$, \ldots, $\Fd_{d-1}(t,u,v,\pi_d(\bfp))$ and in Section \ref{sec:Recurrence} we finish the proof of Theorem \ref{thm:Main} by deriving the recurrence relations from the differentially algebraic system.

\section{Hook expressions and the KP hierarchy} \label{sec:KP}
We recall that an integer partition, or simply a partition, $\lambda = (\lambda_1, \dotsc, \lambda_\ell)$ is a non-increasing sequence of positive integers, $\lambda_1\geq \dotsb\geq\lambda_\ell>0$. The size of $\lambda$ is $|\lambda| = \sum_{i=1}^\ell \lambda_i$ and we also write $\lambda\vdash|\lambda|$. Its length $\ell(\lambda) = \ell$ is the number of parts. We denote $\YY$ the set of all partitions.

Let $\bfp = (p_1, p_2, \dotsc)$ be an infinite set of indeterminates and $R[[\bfp]]$ the ring of formal power series in these variables with coefficients in a ring $R$. If $A\in R[[\bfp]]$, then denote 
\begin{equation*}
	A_{i_1, \dotsc, i_n}(\bfp) = \frac{\partial^{n}A(\bfp)}{\partial p_{i_1} \dotsm \partial p_{i_n}}.
\end{equation*}

The KP hierarchy is an infinite set of PDEs on a function $\tau(\bfp)\in R[[\bfp]]$, called a tau function, or equivalently on $F(\bfp)=\ln \tau(\bfp)$ (if well-defined). The KP hierarchy can also be thought of as an infinite set of \emph{algebraic} equations relating some quantities denoted $H_{i_1, \dotsc, i_n}$ for $n\geq 0$ and $i_1, \dotsc, i_n\geq1$. The integrability property of the hierarchy implies that this point of view is compatible with $H_{i_1, i_2, \dotsc}$ being the derivative of a function $F(\bfp)\in R[[\bfp]]$ with respect to $p_{i_1}, p_{i_2}, \dotsc$ In particular, since $F_{\sigma(i_1), \dotsc, \sigma(i_n)}(\bfp) = F_{i_1, \dotsc, i_n}(\bfp)$ for any permutation of $\{1,\dotsc,n\}$, it is enough to work with the set $\mathcal{H} = \{H_\lambda\}_{\lambda\in\mathbb{Y}}$. In this section, we treat the $H_\lambda$s as formal variables, and denote $\Phi_F: \QQ[\mathcal{H}]\to R[[\bfp]]$ the evaluation of $P(\mathcal{H})\in \QQ[\mathcal{H}]$ obtained by substituting all the $H_\lambda$s with the derivatives $F_\lambda(\bfp)$ of $F(\bfp)$,
\begin{equation*}
	\Phi_F P(\mathcal{H})(\bfp) = P(\mathcal{H})_{|\{H_\lambda = F_\lambda(\bfp)\}_{\lambda\in\YY}}
\end{equation*}

We introduce the \emph{degree} of the variable $H_\lambda$ as $|\lambda|$ and the degree of the product $H_{\lambda^{(1)}} \dotsm H_{\lambda^{(n)}}$ as $\sum_{i=1}^d |\lambda^{(i)}|$. Let $\mathbb{Y}_{\text{hook}}$ the set of \emph{hook partitions}, that is partitions of the form $\lambda=(i, 1^a)$ for $i\geq 1$ and $a\geq0$ (containing at most one part that is not 1), and $\mathcal{H}_{\text{hook}} = \{H_{\lambda}\}_{\lambda \in \mathbb{Y}_{\text{hook}}}$. We say that a \emph{hook expression} is a polynomial in $\QQ[\mathcal{H}_{\text{hook}}]$.

\begin{prop}[Dubrovin-Natanzon \cite{DubrovinNatanzon1989}] \label{prop:Fij}
	For every partition $\lambda \in\YY$, the KP hierarchy provides a hook expression for $F_\lambda(\bfp)$. There exists $\operatorname{KP}_\lambda(\mathcal{H}_{\text{hook}})\in \QQ[\mathcal{H}_{\text{hook}}]$ homogeneous of degree $|\lambda|$ such that 
	\begin{equation*}
		F_\lambda(\bfp) = \Phi_F\operatorname{KP}_\lambda(\mathcal{H}_{\text{hook}})(\bfp),
	\end{equation*}
	and for any hook $\mu=(i,1^j)$ appearing in $\operatorname{KP}_\lambda(\mathcal{H}_{\text{hook}})$, one has $i\leq |\lambda|-\ell(\lambda)+1$.
	
	In particular for $\lambda=(k,l)$ with $k, l\geq 2$ one has
	\begin{equation*}
		\operatorname{KP}_{(k,l)}(\mathcal{H}_{\text{hook}}) = H_{k+l-1,1} + \operatorname{KP}'_{(k,l)}(\mathcal{H}_{\text{hook}})
	\end{equation*}
	where $\operatorname{KP}'_{(k,l)}(\mathcal{H}_{\text{hook}})$ does not contain any $H_{k+l-1,1}$.
\end{prop}

We will later only need the case where $\lambda=(k,l)$ with $k\geq l\geq 2$. This is the version that is stated in Dubrovin-Natanzon, but the two statements are actually equivalent, as will be clear in the proof below. The proof is constructive and provides an algorithm that calculates the polynomials $\operatorname{KP}_\lambda(\mathcal{H}_{\text{hook}})$ iteratively. Here are the polynomials for non-hook partitions with two parts and of size up to 8, calculated with Sagemath.
\begin{equation*}
	\begin{array}{|c|l|}
		\hline
		\lambda & \operatorname{KP}_\lambda(\mathcal{H}_{\text{hook}})\\
		\hline
		(2,2) &  - \frac{1}{2}H_{1^2}^2 - \frac{1}{12}H_{1^4} + H_{3,1}\\
		(3,2) & -H_{1,1}H_{2,1} - \frac{1}{6}H_{2,1^3} + H_{4,1}\\
		(4,2) & \frac{1}{8}H_{1^3}^2 + \frac{1}{12}H_{1^2} H_{1^4} - \frac{1}{2}H_{2,1}^2 - H_{1^2}H_{3,1} + \frac{1}{120}H_{1^6} - \frac{1}{4}H_{3,1^3} + H_{5,1}\\
		(3,3) & \frac{1}{3}H_{1^2}^3 + \frac{1}{4}H_{1^3}^2 + \frac{1}{3}H_{1^2}H_{1^4} - H_{2,1}^2 - H_{1^2}H_{3,1} + \frac{1}{45} H_{1^6} - \frac{1}{3}H_{3,1^3} + H_{5,1}\\
		(5,2) & \frac{1}{6}H_{1^4}H_{2,1} + \frac{1}{2}H_{1^3}H_{2,1^2} + \frac{1}{6}H_{1^2}H_{2,1^3} - H_{2,1}H_{3,1} - H_{1^2}H_{4,1} + \frac{1}{30}H_{2,1^5} - \frac{1}{3}H_{4,1^3} + H_{6,1}\\
		\hline
		(4,3) & 
		H_{1^2}^2H_{2,1} + \frac{1}{2}H_{1^4}H_{2,1} + H_{1^3}H_{2,1^2} + \frac{2}{3}H_{1^2}H_{2,1^3} - 2H_{2,1}H_{3,1} - H_{1^2}H_{4,1} \\
		& + \frac{1}{12}H_{2,1^5} - \frac{1}{2}H_{4,1^3} + H_{6,1}\\
		\hline
		(6,2) & 
		-\frac{1}{4}H_{1,1} H_{1,1,1}^2 - \frac{1}{12} H_{1^2}^2 H_{1^4} - \frac{11}{72} H_{1^4}^2 - \frac{5}{24} H_{1^3} H_{1^5} - \frac{1}{20} H_{1^2} H_{1^6} + \frac{1}{2} H_{2,1^2}^2 \\
		& + \frac{1}{3} H_{2,1} H_{2,1^3} + \frac{1}{4} H_{1^4} H_{3,1} - \frac{1}{2} H_{3,1}^2 + \frac{3}{4} H_{1^3} H_{3,1^2} + \frac{1}{4} H_{1^2} H_{3,1^3} - H_{2,1} H_{4,1} - H_{1^2} H_{5,1} \\
		&- \frac{1}{252} H_{1^8} + \frac{1}{12} H_{3,1^5} - \frac{5}{12} H_{5,1^3} + H_{7,1}\\
		\hline
		(5,3) & -\frac{3}{4} H_{1^2} H_{1^3}^2 - \frac{1}{3} H_{1^2}^2 H_{1^4} + H_{1^2} H_{2,1}^2 + H_{1^2}^2 H_{3,1} - \frac{13}{36} H_{1^4}^2 - \frac{1}{2} H_{1^3} H_{1^5} \\
		& - \frac{2}{15} H_{1^2} H_{1^6} + H_{2,1^2}^2 + H_{2,1} H_{2,1^3} + \frac{2}{3} H_{1^4} H_{3,1} - H_{3,1}^2  + \frac{3}{2} H_{1^3} H_{3,1^2} + H_{1^2} H_{3,1^3} \\
		&- 2 H_{2,1} H_{4,1} - H_{1^2} H_{5,1} - \frac{1}{105} H_{1^8} + \frac{1}{5} H_{3,1^5} - \frac{2}{3} H_{5,1^3} + H_{7,1}\\
		\hline
		(4,4) & -\frac{1}{4} H_{1^2}^4 - \frac{5}{4} H_{1^2} H_{1^3}^2 - \frac{5}{6} H_{1^2}^2 H_{1^4} + 2 H_{1^2} H_{2,1}^2 + H_{1^2}^2 H_{3,1}  - \frac{25}{48} H_{1^4}^2 - \frac{17}{24} H_{1^3} H_{1^5} \\
		&- \frac{13}{60} H_{1^2} H_{1^6} + \frac{5}{4} H_{2,1^2}^2 + \frac{3}{2} H_{2,1} H_{2,1^3} + \frac{3}{4} H_{1^4} H_{3,1} - \frac{3}{2} H_{3,1}^2 + \frac{7}{4} H_{1^3} H_{3,1^2} \\
		&+ \frac{5}{4} H_{1^2} H_{3,1^3} - 2 H_{2,1} H_{4,1} - H_{1^2} H_{5,1} - \frac{23}{1680} H_{1^8} + \frac{1}{4} H_{3,1^5} - \frac{3}{4} H_{5,1^3} + H_{7,1}\\
		\hline
	\end{array}
\end{equation*}

The rest of this section is devoted to the proof of Proposition \ref{prop:Fij} from Hirota's bilinear equation. This follows \cite{DubrovinNatanzon1989} quite closely, but we still include it because Proposition \ref{prop:Fij} does not seem to be known in the combinatorics community and at the same time it is the key formulation that makes the KP hierarchy exploitable for our purposes. We also offer a fair bit more details of the proof, making it hopefully more accessible to combinatorists. We start with the following lemma. Recall the definition of the Hirota derivative. Let $D_i: f\otimes g = f_i\otimes g - f\otimes g_i$ for $f, g\in R[[\bfp]]$. Moreover, let $\operatorname{ev} (f\otimes g)(\bfp) = f(\bfp) g(\bfp)$ be the pointwise product. Then Hirota's derivative is
\begin{equation*}
	D_{i_1} \dotsm D_{i_n} (f\cdot g)(\bfp) \equiv \operatorname{ev} D_{i_1} \dotsc D_{i_n} (f\otimes g)(\bfp)\ \in R[[\bfp]],
\end{equation*}
where the left hand side is the customary notation (see e.g. \cite{Kac2013Bombay}).

Denote $h_k(\bfp)$ the polynomials from the expansion $e^{\sum_{i\geq 1} \frac{p_i}{i} t^i} = \sum_{k\geq 0} h_k(\bfp) t^k \in\QQ[\bfp][[t]]$. They have an explicit expansion in the $p_i$s: $h_k(\bfp) = \sum_{\lambda} \frac{1}{z_\lambda} p_{\lambda_1} \dotsm p_{\lambda_{\ell(\lambda)}}$ where $z_\lambda = \prod_{i\geq1} i^{m_i(\lambda)} m_i(\lambda)!$ where $m_i(\lambda)$ is the number of parts equal to $i$ in $\lambda$. Finally, we denote $\check{D} = (D_1, 2D_2, 3D_3, \dotsc)$.

\begin{lemma}
	Let $k\geq l\geq 2$, and $\tau(\bfp)$ a KP tau function. Then
	\begin{equation} \label{KP-Reduction}
		\Bigl(k h_{l}(\check{D}) D_k + (l-1) h_{k+1}(\check{D}) D_{l-1} - (k+l-1) D_{k+l-1,1}\Bigr)(\tau\cdot\tau)(\bfp) = 0.
	\end{equation}
\end{lemma}

\begin{proof}
	The KP hierarchy for the tau function is encoded as a generating set of equations \cite{Kac2013Bombay} known as Hirota's bilinear equations,
	\begin{equation} \label{UsualKP}
		\sum_{j\geq0} h_j(-2\bfq) h_{j+1}(\check{D}) e^{\sum_{i\geq1} q_i D_i} (\tau\cdot\tau)(\bfp) = 0.
	\end{equation}
	The Hirota derivative $D_\lambda(f\cdot g)(\bfp)\coloneqq D_{\lambda_1} \dotsm D_{\lambda_{\ell(\lambda)}}(f\cdot g)(\bfp)$ satisfies
	\begin{equation*}
		D_\lambda (\tau\cdot\tau)(\bfp) = \sum_{\substack{I=\{i_1, \dotsc, i_r\}\\ J=\{j_1, \dotsc, j_s\}\\ I\sqcup J=[\ell(\lambda)]}} \tau_{\lambda_{i_1}, \dotsc, \lambda_{i_r}}(\bfp) (-1)^{s} \tau_{\lambda_{j_1}, \dotsc, \lambda_{j_s}}(\bfp),
	\end{equation*}
	the sum describing all ways of distributing the parts of $\lambda$ into two disjoint sets (the parts being considered as labeled, so that when two parts are equal and one is in $I$ and the other is in $J$, there is another contribution by exchanging them). Either $I$ or $J$ may be empty. For short we may write
	\[D_\lambda (\tau\cdot\tau)(\bfp) = \sum_{I\sqcup J=[\ell(\lambda)]} \tau_{\lambda_I}(\bfp) (-1)^{|J|} \tau_{\lambda_J}(\bfp),
	\]
	with the notation $\tau_{\lambda_I}(\bfp) = \tau_{\lambda_{i_1}, \dotsc, \lambda_{i_r}}(\bfp)$ whenever $I=\{i_1, \dotsc, i_r\}$.
	
	The next equation we need is traditionally set in the context of symmetric functions. For this, the variables $p_i$s are interpreted as power-sums $p_i(\mathbf{x}) = \sum_{k\geq 1} x_k^i$ over an auxiliary alphabet $\mathbf{x}=(x_1, x_2, \dotsc)$. Then $h_k(\bfp(\mathbf{x})) = \sum_{i_1\leq i_2\leq \dotsb} x_{i_1} x_{i_2} \dotsm$ is the $k$-th complete homogeneous symmetric function. The Cauchy identity gives the following expansion
	\begin{equation*}
		e^{\sum_{i\geq 1}\frac{p_i q_i}{i}} = \sum_{\lambda} h_\lambda(\bfp) m_\lambda(\bfq),
	\end{equation*}
	where $h_\lambda(\bfp) = \prod_k h_{\lambda_k}(\bfp)$ and $m_\lambda(\bfp)$ is the monomial symmetric function $m_\lambda$ expressed as function of the power-sums (in other words $m_\lambda(\bfp(\mathbf{x}))$ is the usual monomial symmetric function). We use the Cauchy identity as follows
	\begin{equation*}
		e^{\sum_{i\geq1} q_i D_i} = \sum_{\mu\in\mathbb{Y}} h_{\mu}(-2\bfq) m_\mu(-\check{D}/2).
	\end{equation*}
	Hence
	\begin{equation*}
		\sum_{\lambda\in\mathbb{Y}} h_\lambda(-2\bfq) \sum_{\substack{\mu\in\mathbb{Y}, j\geq 0\\ \mu\cup(j) = \lambda}} m_\mu(-\check{D}/2) h_{j+1}(\check{D}) \tau(\bfp)\cdot\tau(\bfp) = 0,
	\end{equation*}
	and since $h_\lambda(-2\bfq)$ forms a basis of polynomials in $\bfq$ (or equivalently of symmetric functions in an auxiliary alphabet $\mathbf{x}$ such that $\bfq=\bfp(\mathbf{x})$), we deduce for all $\lambda\in\mathbb{Y}$,
	\begin{equation*}
		\sum_{\substack{\mu\in\mathbb{Y}, j\geq 0\\ \mu\cup(j) = \lambda}} m_\mu(-\check{D}/2) h_{j+1}(\check{D}) (\tau\cdot\tau)(\bfp) = 0.
	\end{equation*}
	We now set $\lambda=(k,l-1)$ for $k>l>1$. There are three contributions: $\mu=(k), j=l-1$ and $\mu=(l-1), j=k$ and $\mu=(k,l-1), j=0$,
	\begin{equation*}
		\Bigl(m_{(l-1)}(-\check{D}/2) h_{k+1}(\check{D}) + m_{(k)}(-\check{D}/2) h_{l}(\check{D}) + m_{(k,l-1)}(-\check{D}/2) h_{1}(\check{D})\Bigr)(\tau\cdot\tau)(\bfp) = 0.
	\end{equation*}
	The monomial symmetric functions in the power sum basis are $m_{(k)}(\bfp) = p_k$ and $m_{(i,j)}(\bfp) = p_ip_j-p_{i+j}$. This gives the lemma.
\end{proof}	
	
\begin{proof}[Proof of Proposition \ref{prop:Fij}]
	First, if $\lambda$ is a hook, then there is nothing to prove. The rest of the proof is a more detailed reformulation of Dubrovin-Natanzon \cite{DubrovinNatanzon1989}.
	
	A non-hook partition $\lambda$ has at least two parts larger than 1, $\lambda_1\geq \lambda_2\geq2$. We proceed by induction on $n(\lambda)\coloneqq|\lambda|-\ell(\lambda)\geq 2$. At $n(\lambda) = 2$ there is only $\lambda=(2,2)$. Then the statement of the proposition is simply the KP equation $F_{2,2}(\bfp) = F_{3,1}(\bfp) - \frac{1}{2}F_{1^2}(\bfp)^2 - \frac{1}{12}F_{1^4}(\bfp)$ which is indeed the evaluation of the hook expression $\operatorname{KP}_{2,2}(\mathcal{H}_{\text{hook}}) = H_{3,1} - \frac{1}{2}H_{1^2}^2 - \frac{1}{12}H_{1^4}$.
	
	Then let $n\geq 2$ and assume that the statement holds for all $2\leq n(\lambda)\leq n$. We now prove the statement for partitions $\lambda$ such that $n(\lambda)=n+1$, by induction on $\ell(\lambda)$ starting with the case $\ell(\lambda)=2$, i.e. $\lambda=(k,l)$ with $k,l\geq2$ and $k+l=n+3$. We write \eqref{KP-Reduction} and use the expansion $h_k(\check{D}) = \sum_{\mu\vdash k} (\prod_{i=1}^{\ell(\mu)} \mu_i)/z_\mu D_\mu$. We start with
	\begin{equation*}
		\begin{aligned}
			&h_{k+1}(\check{D}) D_{l-1} (\tau\cdot\tau)(\bfp) = \sum_{\mu\vdash k+1} \frac{\prod_{i=1}^{\ell(\mu)} \mu_i}{z_\mu} D_{\mu, l-1} (\tau\cdot\tau)(\bfp)\\
			&= \tau(\bfp)^2 \sum_{\substack{\mu\vdash k+1\\ \tilde{\mu}\equiv\mu\cup(l-1)}} \frac{\prod_{i=1}^{\ell(\mu)} \mu_i}{z_\mu} \sum_{I\sqcup J=[\ell(\tilde{\mu})]} (-1)^{|J|} \sum_{X_1\sqcup \dotsb\sqcup X_r = I} \sum_{Y_1\sqcup \dotsb\sqcup Y_s = J} \prod_{i=1}^{r} F_{\tilde{\mu}_{X_i}}(\bfp) \prod_{j=1}^{s} F_{\tilde{\mu}_{Y_j}}(\bfp).
		\end{aligned}
	\end{equation*}
	Here $X_1\sqcup \dotsb\sqcup X_r$ and $Y_1\sqcup \dotsb\sqcup Y_s$ denote set partitions (the $X_i$s and the $Y_j$s are unordered). As a convention, when $I$ is empty, the sum over the set partitions $X_1\sqcup \dotsb\sqcup X_r = I$ can be understood as having a single element (the empty set), and $F_\emptyset(\bfp)\equiv 1$ (and similarly when $J=\emptyset$).
	
	All contributions take the form $\prod_{i=1}^{r} F_{\tilde{\mu}_{X_i}}(\bfp) \prod_{j=1}^{s} F_{\tilde{\mu}_{Y_j}}(\bfp)$ and we have to analyze the values of $n(\tilde{\mu}_{X_i})$ and $n(\tilde{\mu}_{Y_j})$. By symmetry we can consider only $n(\tilde{\mu}_{X_i})$. For any $m\in\{1, \dotsc, r\}$, $n(\tilde{\mu}_{X_m}) \leq \sum_{i=1}^r n(\tilde{\mu}_{X_i})$, so it is enough to study the values taken by this sum,
	\begin{equation*}
		\sum_{i=1}^r n(\tilde{\mu}_{X_i}) = \sum_{i=1}^r |\tilde{\mu}_{X_i}| - \sum_{i=1}^r \ell(\tilde{\mu}_{X_i}).
	\end{equation*}
	There are two cases, depending on whether $\tilde{\mu}_{X_1}\cup \dotsb \cup \tilde{\mu}_{X_r}$ contains the special part $l-1$ or not. If not, then $\sum_{i=1}^r |\mu_{X_i}|\leq k+1$, hence
	\begin{itemize}
		\item if $\sum_{i=1}^r \ell(\tilde{\mu}_{X_i})=1$, then $r=1$ and $\tilde{\mu}_{X_1}$ has a single part and $F_{\tilde{\mu}_{X_1}}(\bfp)$ is given by a hook expression,
		\item if $\sum_{i=1}^r \ell(\tilde{\mu}_{X_i})\geq 2$, then $\sum_{i=1}^r n(\tilde{\mu}_{X_i}) \leq k-1 = n+2-l\leq n$. Therefore by the induction hypothesis, there exists a hook expression for each $F_{\tilde{\mu}_{X_i}}(\bfp)$.
	\end{itemize}
	Now consider the case where $\tilde{\mu}_{X_1}\cup \dotsb \cup \tilde{\mu}_{X_r}$ contains the special part $l-1$, and thus $\sum_{i=1}^r |\mu_{X_i}|\leq k+l$ and
	\begin{itemize}
		\item if $\sum_{i=1}^r \ell(\tilde{\mu}_{X_i})=1$, then $r=1$ and $\tilde{\mu}_{X_1}$ has a single part equal to $l-1$ and $F_{\tilde{\mu}_{X_1}}(\bfp)$ is given by a hook expression,
		\item if $\sum_{i=1}^r \ell(\tilde{\mu}_{X_i})=2$, then either $r=2$ with $\ell(\tilde{\mu}_{X_1}) = \ell(\tilde{\mu}_{X_2}) = 1$ and $F_{\tilde{\mu}_{X_1}}(\bfp), F_{\tilde{\mu}_{X_2}}(\bfp)$ are both given by hook expressions, or $r=1$ with $\ell(\tilde{\mu}_{X_1}) = 2$. In the latter case, the partition $\tilde{\mu}_{X_1}$ is of the form $(\mu_1, l-1)$ (or the other way around) and $n(\tilde{\mu}_{X_1}) = \mu_1 + l-3$. As long as $\mu_1<k+1$, it comes $n(\tilde{\mu}_{X_1}) <n+1$ and by induction there is a hook expression. However if $\mu_1=k+1$, then one gets a contribution $F_{k+1, l-1}(\bfp)$.
		\item if $\sum_{i=1}^r \ell(\tilde{\mu}_{X_i})> 2$, then $\sum_{i=1}^r n(\tilde{\mu}_{X_i}) \leq k+l-3 \leq n$. Therefore by the induction hypothesis, there exists a hook expression for each $F_{\tilde{\mu}_{X_i}}(\bfp)$.
	\end{itemize}
	Overall this gives
	\begin{equation*}
		h_{k+1}(\check{D}) D_{l-1} (\tau\cdot\tau)(\bfp) = \tau(\bfp)^2 \bigl(2F_{k+1, l-1}(\bfp) + \Phi_F\text{hook expression}(\bfp)\bigr).
	\end{equation*}
	(the factor two comes from interchanging the roles of the $X_i$s and the $Y_j$s). Here ``$\text{hook expression}(\bfp)$'' means that all occurences of $H_\lambda$ in the hook expression is substituted with $F_{\lambda}(\bfp)$. Similarly,
	\begin{equation*}
		h_{l}(\check{D}) D_{k} (\tau\cdot\tau)(\bfp) = \tau(\bfp)^2 \bigl(2F_{k, l}(\bfp) + \text{hook expression}(\bfp)\bigr).
	\end{equation*}
	The last term of \eqref{KP-Reduction} is $D_{k+l-1,1} (\tau\cdot\tau)(\bfp) = 2\tau(\bfp)^2 F_{k+l-1, 1}(\bfp)$ which is a hook expression. Overall Equation \eqref{KP-Reduction} reduces to
	\begin{equation} \label{Induction}
		kF_{k,l}(\bfp) = -(l-1)F_{k+1, l-1}(\bfp) + \text{hook expression}(\bfp).
	\end{equation}
	If $l=2$, we interpret it as a hook expression on $\QQ[\mathcal{H}_{\text{hook}}]$ (by replacing all $F_\mu(\bfp)$ with $H_\mu$), and otherwise one performs a simple induction to decrease the value of $l$ down to 2.
	
	Let us prove that $\operatorname{KP}_{(k,l)}(\mathcal{H}_{\text{hook}})$ is homogeneous of degree $k+l$. First, $h_k(\bfp)$ is homogenous of degree $k$ with respect to the degree defined by having $\prod_{i=1}^\ell p_{\lambda_i}$ of degree $|\lambda| = \sum_{i=1}^\ell \lambda_i$. Moreover, it is clear from $D_k \tau\otimes \tau = F_k\tau\otimes \tau - \tau\otimes F_k\tau$ that the Hirota derivative $D_\lambda(\tau\cdot\tau)(\bfp)$ only produces terms of degree $|\lambda|$.
	
	Now we prove that the first part of any hook appearing in $\operatorname{KP}_{(k,l)}(\mathcal{H}_{\text{hook}})$ is less or equal to $n((k,l)) +1 = k+l-1$. From homogeneity, and because there is no negative contribution to the degree, we get the bound $k+l$. However $H_{k+l}$ would have to come from $D_{k+l}(\tau\cdot\tau)(\bfp)$, which actually vanishes for any $\tau(\bfp)$ because it has an odd number of Hirota derivatives. The largest first part of any $H_\mu$ in $\operatorname{KP}_{(k,l)}(\mathcal{H}_{\text{hook}})$ is thus $k+l-1$.
	
	One can further easily track the coefficient of $H_{k+l-1,1}$ in the final hook expression. Indeed, by making explicit the contribution of $D_{k+l-1,1} (\tau\cdot\tau)(\bfp)$ to \eqref{Induction}, the latter becomes $kF_{k,l}(\bfp) = -(l-1)F_{k+1, l-1}(\bfp) + (k+l-1)F_{k+l-1,1}(\bfp) + \text{hook expression}'(\bfp)$, where the prime indicates that the hook expression does not contain any $F_{k+l-1,1}(\bfp)$. In the special case $l=2$, this gives $F_{k,l}(\bfp) = F_{k+l-1,1}(\bfp) + \text{hook expression}'(\bfp)$ and by induction one finds that the coefficient is always 1.
	
	Finally, one performs an induction on $\ell(\lambda)$ at fixed $n(\lambda)$. Let $\lambda=(\lambda_1, \dotsc, \lambda_\ell)$ and $\lambda'=(\lambda_1, \dotsc, \lambda_{\ell-1})$. Since $n(\lambda')\leq n(\lambda)$, there exists $\operatorname{KP}_{\lambda'}(\mathcal{H}_{\text{hook}})$ such that $F_{\lambda'}(\bfp) = \Phi_F \operatorname{KP}_{\lambda'}(\mathcal{H}_{\text{hook}})(\bfp)$. As an expression in $R[[\bfp]]$, the RHS can be differentiated  with respect to $p_{\lambda_\ell}$. Note that $\operatorname{KP}_{\lambda'}(\mathcal{H}_{\text{hook}})(\bfp)$ has no explicit dependence on $p_\ell$. Therefore substituting back $H_\mu$ for $F_{\mu}(\bfp)$ in $P(\bfp)\equiv \frac{\partial}{\partial p_{\lambda_\ell}}\Phi_F \operatorname{KP}_{\lambda'}(\mathcal{H}_{\text{hook}})(\bfp)$ gives a polynomial in $\mathcal{H}$.
	
	If $\lambda_\ell=1$ this results directly in a hook expression. If not, $P(\bfp)$ is found to have terms $F_{\mu\cup \lambda_\ell}(\bfp)$ where $\mu=(i,1^j)$ is a a hook. By induction, $i\leq n(\lambda')+1$ then
	\begin{equation*}
		n((i\cup \lambda_\ell)) = i+\lambda_{\ell} - 2 \leq n(\lambda).
	\end{equation*}
	Therefore the induction assumption gives $\operatorname{KP}_{(i\cup\lambda_\ell)}(\mathcal{H}_{\text{hook}})$ such that 
	\begin{equation*}
		F_{(i\cup\lambda_\ell)}(\bfp) = \Phi_F \operatorname{KP}_{(i\cup\lambda_\ell)}(\mathcal{H}_{\text{hook}})(\bfp).
	\end{equation*}
	It comes that 
	\begin{equation*}
		F_{(\mu\cup\lambda_\ell)}(\bfp) = \frac{\partial^j}{\partial p_1^j} \Phi_F \operatorname{KP}_{(i\cup\lambda_\ell)}(\mathcal{H}_{\text{hook}})(\bfp) = \Phi_F \operatorname{KP}_{(\mu\cup\lambda_\ell)}(\mathcal{H}_{\text{hook}})(\bfp)
	\end{equation*}
	where the second equality defines the polynomial $\operatorname{KP}_{(\mu\cup\lambda_\ell)}(\mathcal{H}_{\text{hook}})$. Finally we check that the largest part appearing in $\operatorname{KP}_{(\mu\cup\lambda_\ell)}(\mathcal{H}_{\text{hook}})$ is $i+\lambda_\ell-1\leq n(\lambda)+1$.
\end{proof}

\section{Bipartite maps} \label{sec:Maps}

From now on, $F(t,u,v,\bfp) = \ln \tau(t,u,v,\bfp)$ is the generating function of connected, edge-labeled, bipartite maps, counted by number of edges with $t$, number of white vertices with $u$, black vertices with $v$ and number of faces of degree $i$ with $p_i$, and $F_{\lambda_1, \lambda_2, \dotsc}(t,u,v,\mathbf{p}) = \frac{\partial F}{\partial p_{\lambda_1} \partial p_{\lambda_2} \dotsm}(t,u,v,\mathbf{p})$. Our analysis relies on one side on the following theorem.

\begin{thm}[Goulden-Jackson \cite{GouldenJackson2008}]
	$\tau(t,u,v,\bfp)$ is a KP tau function.
\end{thm}

On the other side, the key equations which actually determine the tau function are the Virasoro constraints. Recall that $\pi_d$ is the map that evaluates functions on $\pi_d(\bfp) = (p_1, \dotsc, p_d, 0, 0, \dotsc)$, i.e. $\pi_d(p_i) = 0$ for $i>d$, and $F^{\pi_d}_\lambda\equiv \Fd_\lambda(t,u,v,\pi_d(\bfp)) \coloneqq {\pi_d}F_\lambda(t,u,v,\bfp) \in\mathbb{K}_d[[t]]$. We also denote $\pi_d\mathcal{F} = \{\Fd_\lambda\}_{\lambda\in\YY}$ and similarly for $\pi_d\mathcal{F}_{\text{hook}}$. From here on out, \emph{we will stop distinguishing the formal variables $H_\lambda$s and their specialization to $\Fd_\lambda$s}. We hope no confusion will arise from this choice.

\subsection{Virasoro constraints}
\begin{thm}[Virasoro constraints -- Folklore] \label{thm:Virasoro}
	For all $k\geq 0$ and $\lambda = (\lambda_1, \dotsc, \lambda_\ell)\in\YY$, the set $\{F^{\pi_d}_\lambda\}_{\lambda\in\YY}$ satisfies the following set of algebraic equations known as Virasoro constraints
	\begin{multline} \label{AllVirasoro}
		(k+1)\Fd_{\lambda,k+1} = t\sum_{\substack{i,j\geq 1\\i+j=k}} ij\Bigl(\Fd_{\lambda,i,j} + \sum_{I\sqcup J=\{1, \dotsc,\ell\}} \Fd_{\lambda_I,i} \Fd_{\lambda_J,j}\Bigr) + t\sum_{i=1}^d (i+k)p_i \Fd_{\lambda,i+k} \\+ t\sum_{s=1}^\ell (\lambda_s + k)\Fd_{\lambda-\lambda_s,\lambda_s+k} + t(u+v)k\Fd_{\lambda,k} + tuv \delta_{k,0} \delta_{\lambda,\emptyset}.
	\end{multline}
	Here we have denoted $\lambda_I = (\lambda_{i_1}, \dotsc, \lambda_{i_a})$ for a list $I = (i_1\geq \dotsb\geq i_a)\subset \{1, \dotsc, \ell\}$ the partition obtained by keeping only the parts of $\lambda$ with indices in $I$, and $I\sqcup J$ is a sum over set bipartitions of $\{1,\dotsc, \ell\}$. Finally, we denote $\lambda-\lambda_s$ the partition $\lambda$ whose part $\lambda_s$ is removed.
\end{thm}

For reference, let us write the constraints for $\lambda=\emptyset$,
	\begin{equation} \label{Virasoro}
	(k+1)\Fd_{k+1} = t\sum_{\substack{i,j\geq 1\\i+j=k}} ij(\Fd_i \Fd_j + \Fd_{i,j}) + t\sum_{i=1}^d (i+k)p_i \Fd_{i+k} + t(u+v)k\Fd_k + tuv \delta_{k,0}.
	\end{equation}	
	
\begin{proof}
	Most references deal with general maps, as opposed to bipartite maps, but all methods work equally in this case. The earliest reference is Bender and Canfield \cite{BenderCanfield1986}. The names ``Virasoro constraints'' and ``loop equations'' originate from the physics literature, where these equations are also well-known from the representation of $\tau(t,u,v,\bfp)$ as a formal matrix integral \cite{AmbjornKristjansenMakeenko} (and see \cite{Eynard:book} for the method -- these equations are also known as loop equations). 
	
	Here we use the notation $p_k^*=k\frac{\partial}{\partial p_k}$ and we sketch the proof of Proposition \ref{thm:Virasoro} combinatorially, following \cite{BenderCanfield1986} but in the bipartite case. The Virasoro constraints are equivalent to a linear equation,
	\begin{equation} \label{VirasoroTau}
		p_{k+1}^*\tau(t,u,v,\bfp) = t\biggl(\sum_{i+j=k} p_i^* p_j^* + \sum_{i\geq 1} p_i p_{i+k}^* + t(u+v)p_k^* + tuv \delta_{k,0}\biggr)\tau(t,u,v,\bfp).
	\end{equation}
	Notice that $p_{k+1}^*\tau(t,u,v,\bfp)$ picks up a face of degree $k+1$, removes its weight $p_{k+1}$, and multiplies by the number $k+1$ of white corners of the face. Therefore, the left hand side of \eqref{VirasoroTau} counts the number of rooted bipartite maps $\mathfrak{m}$ (connected or not) such that the root face has degree $k+1$ (the degree of a face is the number of incident white, or black, corners).
	
	For the RHS of \eqref{VirasoroTau}, we consider the map $\mathfrak{m}'$ obtained from $\mathfrak{m}$ by first orienting the root clockwise then removing the edge $e$ that follows the root clockwise and removing the isolating vertices if nay. It has one less edge, and we consider the effect of adding back $e$ on the faces of $\mathfrak{m}'$. The possibilities are the following.
	\begin{itemize}
		\item The root component of $\mathfrak{m}'$ is empty, meaning that $\mathfrak{m}$ has just one edge two vertices, and a face of degree 1, hence $k=0$, and total weight $tuv$,
		\item $e$ connects a vertex of $\mathfrak{m}'$ to a leaf in $\mathfrak{m}$. In terms of generating function it is $t(u+v) p_k^*\tau(t,u,v,\bfp)$.
		\item $e$ connects two corners in the same face. Therefore it splits this face of $\mathfrak{m}'$ of degree $i+k$ for some $i\geq 1$ into a face of degree $i$ and another of degree $k+1$. In terms of generating function it is $t\sum_{i\geq 1} p_i p_{i+k}^*\tau(t,u,v,\bfp)$.
		\item $e$ connects two corners of distinct faces in $\mathfrak{m}'$, whose degrees are $i$ and $j$ such that $i+j=k$. In terms of generating function it is $t\sum_{i+j=k} p_i^* p_j^*\tau(t,u,v,\bfp)$.
	\end{itemize}
\end{proof}


Although we will need to be more specific, we want to point out the following proposition.
\begin{prop} \label{prop:VirasoroReduction}
	Let $\bar{\mathbb{K}}_d = \QQ[t,t^{-1},u,v,\pi_d(\bfp), p_d^{-1}]$ and
	\begin{equation*}
		\mathcal{F}_{<d} \coloneqq \{\Fd_\lambda| \forall i\in\{1, \dotsc, \ell(\lambda)\}\ \lambda_i\leq d-1\}.
	\end{equation*}
	Then for any partition $\mu$, $\Fd_\mu \in \bar{\mathbb{K}}_d[\mathcal{F}_{<d}]$.
\end{prop}

\begin{proof}
	The proof is an induction on $|\mu|$. If $\Fd_\mu$ has a part larger than $d-1$, say $\mu_a = k+d$ with $k\geq0$, then $\mu=\lambda\cup (k+d)$ and $\Fd_\mu$ can be extracted from \eqref{AllVirasoro} from the term $t\sum_{i=1}^d (i+k)p_i \Fd_{\lambda,i+k}$ with $i=d$. This expresses $\Fd_\mu$ in terms of $\Fd_{\nu \cup k'}$s where $\nu$ contains a subset of the parts of $\mu$ except for $\mu_a$, and $k'<\mu_a$. All these partitions have size less than $\mu$, since $|\nu| + k'<|\mu|$. Then one can perform the induction.
\end{proof}

Notice that there is explicit factor $t$ on the RHS of the Virasoro constraints, which translates the fact that the combinatorial decomposition is based on removing an edge. This explains why $t^{-1}$ is needed in Proposition \ref{prop:VirasoroReduction}.

\subsection{Strategy} Another important equation is the following.
\begin{prop}
	One has
	\begin{equation} \label{Homogeneity}
		t\frac{\partial F}{\partial t}(t,u,v,\bfp) = \sum_{i\geq 1} ip_i F_i(t,u,v,\bfp).
	\end{equation}
\end{prop}

\begin{proof}
	It comes from a double counting of the number of edges. On the LHS it is a direct counting using $t\frac{\partial}{\partial t}$ (recall that $t$ tracks the number of edges in $\tau(t,u,v,\bfp)$, hence in $F(t,u,v,\bfp)$ too). On the RHS, $\sum_{i\geq 1} ip_i \frac{\partial}{\partial p_i}$ counts the number of corners of faces incident to white vertices, which are in bijection with edges. In other words, if $\lambda$ is a partition, then the coefficient of $p_{\lambda_1} p_{\lambda_2}\dotsm$ in $F(t,u,v,\bfp)$ is non-zero only if $t$ has exponent $|\lambda|$.
\end{proof}

The idea to build our ODEs is as follows. Our ODEs come from the equation \eqref{Homogeneity} which we differentiate with respect to $p_k$ for $k=2, \dotsc, d$ and specialize to $\pi_d$. It gives the following equations
\begin{equation*}
	t\frac{\partial}{\partial t} \Fd_k{} - k\Fd_k = \sum_{i=1}^{d} ip_i \Fd_{i,k}.
\end{equation*}
Then on the RHS we use the KP hierarchy to re-express $\Fd_{i,k}$ in terms of $\Fd_\lambda$ where $\lambda$ is a hook. More precisely, this produces hook partitions where the first part goes up to $2d-1$. We thus use the Virasoro constraints to eliminate $\Fd_{d,1^l}, \dotsc, \Fd_{2d-1,1^{l}}$ in favor of $\Fd_{1,1^l}, \dotsc, \Fd_{d-1, 1^l}$. While this re-introduces some non-hook partitions, the latter are eliminated with another round of the KP hierarchy and no more substitution is needed. The functions $\Fd_{1,1^l}, \dotsc, \Fd_{d-1, 1^l}$ remain the sole unknown functions of the analysis.

All these manipulations can actually be performed for other models of maps that satisfy both some Virasoro-type constraints and the KP hierarchy, like general (non-bipartite) maps. However, the last step we need to obtain ODEs is specific to bipartite maps. The Virasoro constraints indeed allow us to recast derivatives with respect to $p_1$ as derivatives with respect to the size parameter $t$, as desired.

To make these manipulations formal, we introduce the notion of $d$-admissible expressions, which are essentially polynomials in $\Fd_{1,1^l}, \dotsc, \Fd_{d-1, 1^l}$.

\subsection{$d$-admissible expressions}
\begin{defi}
	Let 
	\begin{equation*}
		\mathcal{F}_{\text{$d$-adm.}}=\pi_d\mathcal{F}_{\text{hook}}\cap\mathcal{F}_{<d} = \{F^{\pi_d}_\lambda| \text{$\lambda\in\YY$ is a hook and $\lambda_1 \leq d-1$}\}
	\end{equation*}
	A $d$-admissible expression is an element of $R_d\coloneqq\QQ[t, t^{-1}, u,v,\pi_d(\bfp),p_d^{-1},\mathcal{F}_{\text{$d$-adm.}}]$.
\end{defi}

We define a set of $d$-admissible expressions 
\begin{equation} \label{Glarge}
	\mathcal{G}^{\pi_d} \coloneqq \{G_{\lambda}(t,u,v,\pi_d(\bfp),\mathcal{F}_{\text{$d$-adm.}}) \in R_d| \text{$\lambda$ is a hook and  $1\leq \lambda_1\leq 2d-1$}\}.
\end{equation}
For $l\geq0$, we set first for $m=1, \dotsc, d-1$, $G_{m,1^l}(\mathcal{F}_{\text{$d$-adm.}}) \coloneqq \Fd_{m,1^l}$, and then for $k=0,\dotsc, d-1$ we define inductively
\begin{multline} \label{G}
	G_{d+k,1^l}(\mathcal{F}_{\text{$d$-adm.}}) \coloneqq	\frac{1}{(d+k)p_d}\Bigl[- \sum_{\substack{i,j\geq 1\\ i+j=k}} ij\Bigl(\sum_{\substack{r, s\geq0\\ r+s=l}} \binom{l}{r}\Fd_{i,1^r} \Fd_{j,1^s} + \pi_d\operatorname{KP}_{i,j,1^l}(\mathcal{F}_{\text{hook}})\Bigr)\\
	+ (k+1)t^{-1}G_{k+1,1^l}(\mathcal{F}_{\text{$d$-adm.}}) - (u+v)k\Fd_{k,1^l} - l(k+1) G_{k+1,1^{l-1}}(\mathcal{F}_{\text{$d$-adm.}}) - uv\delta_{k,0}\delta_{l,0} \\- \sum_{i=1}^{d-k-1} (i+k)p_i \Fd_{i+k,1^l}
	- \sum_{m=0}^{k-1} (d+m)p_{d+m-k} G_{d+m,1^l}(\mathcal{F}_{\text{$d$-adm.}})\Bigr],
\end{multline}
with the shorthand notation $G_\lambda(\mathcal{F}_{\text{$d$-adm.}}) \equiv G_\lambda(t,u,v,\pi_d(\bfp), \mathcal{F}_{\text{$d$-adm.}})$.
	
\begin{prop} \label{prop:G} The Virasoro constraints \eqref{Virasoro} imply that for $m=1, \dotsc, 2d-1$ and $l\geq 0$
	\begin{equation} \label{G=F}
		\Fd_{m,1^l} = G_{m,1^l}(\mathcal{F}_{\text{$d$-adm.}}).
	\end{equation}
\end{prop}

Obviously the equality \eqref{G=F} is non-trivial only for $m=d, \dotsc, 2d-1$. In the second line of \eqref{G}, notice that for all $k=1, \dotsc, d-2$ we simply have $G_{k+1,1^l}(\mathcal{F}_{\text{$d$-adm.}}) = \Fd_{k+1,1^l}$ and $G_{k+1, 1^{l-1}}(\mathcal{F}_{\text{$d$-adm.}}) = \Fd_{k+1,1^{l-1}}$. It is only for $k=d-1$ that we can not use $\Fd_{d,1^l}$ because it is not $d$-admissible.



\begin{proof} We prove \eqref{G=F} by induction on $k\geq 0$. First consider the Virasoro constraint \eqref{AllVirasoro} for $k=0$ and $\lambda=(1^l)$. It reads
	\begin{equation*}
		\Fd_{1^{l+1}} = t\sum_{i=1}^{d} ip_i \Fd_{i,1^l} + lt\Fd_{1^l} + tuv \delta_{l,0}. 
	\end{equation*}
	We can extract $\Fd_{d,1^l} = (\Fd_{1^{l+1}} - t\sum_{i=1}^{d-1} ip_i \Fd_{i,1^l} - lt\Fd_{1^l} - tuv \delta_{l,0})/(tdp_d)$, which is the definition of $ G_{d,1^l}$ in \eqref{G}. 
	
	Let $k\in\{1, \dotsc, d-1\}$ and assume that \eqref{G=F} holds as $G_{d+k',l} = tp_d^{k'+1} \Fd_{d+k',1^l}$ for all $l\geq 0 $ and $k'<k$. We now consider the Virasoro constraint \eqref{AllVirasoro} for $k$ and $\lambda=(1^l)$,
	\begin{multline}
		(k+1)\Fd_{k+1,1^l} = t\sum_{\substack{i,j\geq 1\\ i+j=k}} ij\Bigl(\sum_{\substack{r, s\geq0\\ r+s=l}} \binom{l}{r} \Fd_{i,1^r} \Fd_{j,1^s} + \Fd_{i,j,1^l}\Bigr) + t\sum_{i=1}^d (i+k)p_i \Fd_{k+i,1^l} \\
		+ tl(k+1) \Fd_{k+1,1^{l-1}} + t(u+v) k\Fd_{k,1^l}.
	\end{multline}
	We will extract $t(d+k) p_d \Fd_{d+k, 1^l}$ from this equation while expressing all other quantities as $d$-admissible ones. One splits the sum $t\sum_{i=1}^d (i+k)p_i \Fd_{k+i,1^l}$ at $k+i=d$,
	\begin{equation*}
		t\sum_{i=1}^d (i+k)p_i \Fd_{k+i,1^l} = t\sum_{i=1}^{d-1-k} (i+k)p_i \Fd_{k+i,1^l} + t\sum_{i=n-k}^{d-1} (i+k)p_i \Fd_{k+i,1^l}.
	\end{equation*}
	 For $i< d-k$, $t\sum_{i=1}^{d-1-k} (i+k)p_i \Fd_{k+i,1^l}$ is $d$-admissible by definition. For $i\geq d-k$, the sum involves $\Fd_{d,1^l}, \dotsc, \Fd_{d+k-1, 1^l}$ which are not $d$-admissible, but for which the induction hypothesis can be applied. 
	 
	 Then, the only remaining quantities that are not $d$-admissible are $\Fd_{i,j,1^l}$ when $(i,j,1^l)$ is not a hook, i.e. when both $i, j\geq2$. For this, we use the KP flows of Proposition \ref{prop:Fij},
	 \begin{equation*}
	 	\Fd_{i,j,1^l} = \operatorname{KP}_{i,j,1^l}(\pi_d\mathcal{F}_{\text{hook}}).
	 \end{equation*}
	 The term with the largest first part in $\operatorname{KP}_{i,j,1^l}(\pi_d\mathcal{F}_{\text{hook}})$ is $\Fd_{i+j-1,1^{l+1}}$ which is $d$-admissible. All other monomials are products over hooks whose first part is smaller, hence they are $d$-admissible too.
\end{proof}

It will be crucial for us to control the explicit powers of $t$ that appear in $G_{d+m,1^l}(\mathcal{F}_{\text{$d$-adm.}})$ and in particular in front of the terms that are linear in $\mathcal{F}_{\text{$d$-adm.}}$. This behavior is infered from the definition \eqref{Glarge}.

\begin{prop}  \label{prop:ExplicitG} Let $l\geq 0$. There exists a finite number of non-zero polynomials $\alpha^{(m,l)}_i$, $\beta^{(m,l)}_i$, $\gamma^{(m,l)}_i$, $\rho^{(m,l)}_{i,r}$, $\theta^{(m,l)}_{\substack{j_1, \dotsc, j_a\\ q_1, \dotsc, q_a}} \in \QQ[u,v,p_d^{-1}, \pi_d(\mathbf{p})]$ such that for $m=0,\dotsc, d-2$
	\begin{multline} \label{Gorder}
		G_{d+m,1^l}(\mathcal{F}_{\text{$d$-adm.}}) = \sum_{i=1}^{m+1} \Bigl(\alpha^{(m,l)}_{i} t^{-1}\Fd_{i,1^l} + \beta^{(m,l)}_{i} \Fd_{i,1^{l-1}}\Bigr) +
		\sum_{i=1}^{d-1} \gamma^{(m,l)}_i \Fd_{i,1^l} \\
		+ \mathbf{1}_{m\geq 2}\sum_{i=1}^m \sum_{r=1}^{i-1} \rho^{(m,l)}_{i,r} \Fd_{i-r,1^{l+r}}
		+ \mathbf{1}_{m\geq 2} \sum_{a\geq 2} \sum_{\substack{j_1, \dotsc, j_a \geq 1\\ \sum_{b=1}^a j_b\leq m}} \sum_{\substack{q_1, \dotsc, q_a\geq 0\\ \sum_{b=1}^a q_b\geq l\\ \sum_{b=1}^a (j_b+q_b) \leq m+l}} \theta^{(m,l)}_{\substack{j_1, \dotsc, j_a\\ q_1, \dotsc, q_a}} \prod_{b=1}^a \Fd_{j_b,1^{q_b}}\\
		+ Q_{m,1^l}(t, u, v, p_d^{-1}, \pi_d(\mathbf{p})) 
	\end{multline}
	and for $m=d-1$
	\begin{multline} \label{G2d-1}
		G_{2d-1,1^l}(\mathcal{F}_{\text{$d$-adm.}}) = \sum_{i=1}^{d-1} \Bigl(\alpha^{(d-1,l)}_{i} t^{-1}\Fd_{i,1^l} + \beta^{(d-1,l)}_i \Fd_{i,1^{l-1}} + \gamma^{(d-1,l)}_i \Fd_{i,1^l}\Bigr) \\
		+ \sum_{i=1}^{d-1} \sum_{r=1}^{i-1} \rho^{(d-1,l)}_{i,r} \Fd_{i-r,1^{l+r}}
		+ \sum_{a\geq 2} \sum_{\substack{j_1, \dotsc, j_a \geq 1\\ \sum_{b=1}^a j_b\leq d-1}} \sum_{\substack{q_1, \dotsc, q_a\geq 0\\ \sum_{b=1}^a q_b\geq l\\ \sum_{b=1}^a (j_b+q_b) \leq d-1+l}} \theta^{(d-1,l)}_{\substack{j_1, \dotsc, j_a\\ q_1, \dotsc, q_a}} \prod_{b=1}^a \Fd_{j_b,1^{q_b}}\\
		+ \frac{\Fd_{1^{l+1}}}{(2d-1)p_d^2t^2} - \frac{2l\Fd_{1^l}}{(2d-1)p_d^2t} + Q_{d-1,1^l}(t, u, v, p_d^{-1}, \pi_d(\mathbf{p})),
	\end{multline}
	where $Q_{m,1^l}(t, u, v, p_d^{-1}, \pi_d(\mathbf{p}))\in \QQ[t, u, v, p_d^{-1}, \pi_d(\mathbf{p})]$ and $\beta_i^{(m,0)}=0$ for $m=0, \dotsc, d-1$.		
\end{prop}

\begin{proof}
	It is a lengthy but straightforward induction on $m$ which is left to the reader. Notice that in \eqref{G} the only terms with an explicit power of $t$ are $\frac{(m+1)}{(d+m)p_dt}\Fd_{m+1,1^l}$ for $m=0, \dotsc, d-2$, which are $d$-admissible, and $\frac{d}{(2d-1)p_dt}G_{d,1^l}$ for $m=d-1$, which is not $d$-admissible. Therefore when one builds the $d$-admissible expression for $G_{d+m,1^l}$ in a recursive manner, an additional substitution is required for $G_{d,1^l}$, using 
	\begin{equation*}
		G_{d,1^l}(\mathcal{F}_{\text{$d$-adm.}}) = \frac{1}{dp_d}\bigl(t^{-1} \Fd_{1^{l+1}} - l\Fd_{1^l} - \sum_{i=1}^{d-1} ip_i \Fd_{i,1^l} - uv \delta_{l,0}\bigr).
	\end{equation*}
	This explains the two terms $\frac{\Fd_{1^{l+1}}}{(2d-1)p_d^2t^2} - \frac{2l\Fd_{1^l}}{(2d-1)p_d^2t}$ in the RHS of \eqref{G2d-1}.
\end{proof}


We need a last set of $d$-admissible expressions.
\begin{prop} \label{prop:Fnk}
	We have the following $d$-admissible expressions. For $j=2, \dotsc, d-1$
	\begin{equation} \label{Fnk}
		\sum_{i=1}^{d} ip_iF^{\pi_d}_{i,j} = \frac{\Fd_{j,1}}{t} - jF^{\pi_d}_j,
	\end{equation}
	and
	\begin{equation} \label{Fnn}
		\sum_{i=1}^{d} ip_iF^{\pi_d}_{i,n} = \frac{\Fd_{1,1}-t(d+1)\Fd_1}{t^2dp_d} + \frac{uv}{p_d} - \sum_{i=1}^{d-1} \frac{ip_i}{dp_d} \Bigl(\frac{\Fd_{i,1}}{t} - d\Fd_i\Bigr)
	\end{equation}
\end{prop}

\begin{proof}
	Consider the Virasoro constraint \eqref{AllVirasoro} for $k=0$ and $\lambda=(j)$. This proves Equation \eqref{Fnk}, and also gives this equation for $j=d$,
	\begin{equation*}
		t\sum_{i=1}^{d} ip_iF^{\pi_d}_{d,i} = \Fd_{d,1} - tdF^{\pi_d}_d.
	\end{equation*}
	In order to obtain \eqref{Fnn}, we use Proposition \ref{prop:G} which gives $\Fd_{d,1^l} = G_{d,1^l}(\mathcal{F}_{\text{$d$-adm.}})$, and use the $d$-admissible expression of $G_{d,1^l}(\mathcal{F}_{\text{$d$-adm.}})$ for $l=0,1$.
\end{proof}

\section{A differentially algebraic system} \label{sec:ODEs}
\subsection{The coupled ODEs} 
Now the key point is that if $i,j\geq 2$ and $i+j\leq 2d$, then $\operatorname{KP}_{i,j,1^l}(\mathcal{H}_{\text{hook}})$ is a polynomial in the $H_\lambda$s where $\lambda$ is a hook whose first part $\lambda_1$ is bounded as $\lambda_1\leq 2d-1$. As a consequence, after applying $\Phi_{\Fd}$, they all possess $d$-admissible expressions in terms of $\mathcal{G}^{\pi_d}$ and it makes sense to substitute them into $\operatorname{KP}_{i,j,1^l}(\mathcal{H}_{\text{hook}})$,
\begin{equation} \label{piKP}
	\Phi_{\Fd}\operatorname{KP}_{i,j,1^l}(\mathcal{H}_{\text{hook}}) = \operatorname{KP}_{i,j,1^l}(\mathcal{G}^{\pi_d}).
\end{equation}
Defining for $k=2, \dots, d-1$
\begin{align}
	&\mathcal{E}_k(\mathcal{G}^{\pi_d}) \equiv \mathcal{E}_k(t,u,v,\pi_d(\mathbf{p}), \mathcal{G}^{\pi_d}) \coloneqq -\sum_{i=1}^{d} ip_i \operatorname{KP}_{i,k}(\mathcal{G}^{\pi_d}) + \frac{\Fd_{k,1}}{t} - kF^{\pi_d}_k\\
	&\begin{multlined}
		\mathcal{E}_1(\mathcal{G}^{\pi_d}) \equiv \mathcal{E}_1(t,u,v,\pi_d(\mathbf{p}), \mathcal{G}^{\pi_d}) \coloneqq -\sum_{i=1}^{d} ip_i\operatorname{KP}_{i,d}(\mathcal{G}^{\pi_d}) + \frac{\Fd_{1,1}-t(d+1)\Fd_1}{t^2dp_d} \\+ \frac{uv}{p_d} - \sum_{i=1}^{d-1} \frac{ip_i}{dp_d} \Bigl(\frac{\Fd_{i,1}}{t} - d\Fd_i\Bigr)
	\end{multlined}
\end{align}
which are elements of $\mathbb{Q}[t,t^{-1},u,v,\pi_d(\mathbf{p}), p_d^{-1}, \mathcal{G}^{\pi_d}]$, we find via Proposition \ref{prop:Fnk} the relations 
\begin{equation*}
	\mathcal{E}_k(\mathcal{G}^{\pi_d}) = 0,
\end{equation*}
for $k=1, \dotsc, d-1$.

The final step of our substitution process is to replace derivatives with respect to $p_1$ with derivatives with respect to $t$ and thus get ODEs. This is done by combining \eqref{Homogeneity}, $t\dt F = \sum_{i\geq 1} ip_i F_i$, with the Virasoro constraint \eqref{Virasoro} for $k=0$, $F_1 = t\sum_{i\geq 1} ip_i F_i + tuv$. This gives
\begin{equation*}
	t^2 \dt F = F_1 - tuv,
\end{equation*}
and with further derivatives
\begin{equation} \label{p1_t}
	\Fd_{k,1^l} = \bigl(t^2 \dt\bigr)^{l} \Fd_k.
\end{equation}
We will now display the $t$-dependence explicitly to emphasize that we are arriving at a system of ODEs with respect to the variable $t$ only. In particular we write $\Fd_k{}'(t) \equiv \dt \Fd_k$, $\Fd_k{}^{(l)} = \dt^l \Fd_k$ and so on. Let $T$ denote the map that replaces all occurences of $p_1$-derivatives in $d$-admissible expressions according to \eqref{p1_t}.

\begin{thm} \label{thm:ODEs} For $k=2, \dotsc, d-1$, let
	\begin{equation} \label{ODEk}
		E_k(t) \coloneqq T\mathcal{E}_k(\mathcal{G}^{\pi_d}) = -\sum_{i=1}^{d} ip_i T\Bigl(\operatorname{KP}_{i,k}(\mathcal{G}^{\pi_d})\Bigr) + t F^{\pi_d}_k{}'(t) - kF^{\pi_d}_k(t),
	\end{equation}
	and
	\begin{multline} \label{ODEF}
		E_1(t) \coloneqq dp_d T\mathcal{E}_1(\mathcal{G}^{\pi_d}) = -dp_d \sum_{i=1}^{d} ip_i T\Bigl(\operatorname{KP}_{i,d}(\mathcal{G}^{\pi_d})\Bigr) + \Fd_1{}'(t) - (d+1)t^{-1}\Fd_1(t) \\+ \sum_{i=1}^{d-1} ip_i (dF^{\pi_d}_i(t)-tF^{\pi_d}_i{}'(t)) -duv.
	\end{multline}
	Then the equations $E_k(t) = 0$ form a system of $d-1$ coupled ODEs satisfied by $\Fd_1(t)$, $\Fd_2(t)$, \ldots, $\Fd_{d-1}(t)$.
\end{thm}

\begin{proof}
	The fact that these ODEs hold is a direct result of Propositions \ref{prop:Fij}, \ref{prop:Fnk}, \ref{prop:G} and Equations \eqref{piKP}, \eqref{p1_t}. By construction, they only involve $\Fd_1(t), \Fd_2(t), \dotsc, \Fd_{d-1}(t)$ and their $t$-derivatives, since the equations are built from $d$-admissible expressions on which the map $T$ is applied.
\end{proof}

\subsection{Applications} Here we give a few special cases of Theorem \ref{thm:ODEs} obtained with Sagemath.
\subsubsection{Quadrangulations with digons} For $d=2$ we find for $E_1(t)=0$,
\begin{multline}
	t  \Fd_1{}'(t) +  \Fd_1(t) = p_2^2 t^7  \Fd_1{}'''(t) + 6 p_2^2 t^5  \Fd_1{}'(t)^2 + 6 p_2^2 t^6  \Fd_1{}''(t) + 6 p_2^2 t^5  \Fd_1{}'(t) \\- p_1^2 t^3  \Fd_1{}'(t) + 4 p_2 t^3 (u+v)  \Fd_1{}'(t) + 2 p_1 t^2  \Fd_1{}'(t) + 2 t u v + p_1 t  \Fd_1(t) 
\end{multline}
which is an elementary generalization of the ODE from Carrell-Chapuy \cite{CarrellChapuy2015} to non-zero $p_1$, i.e. with additional digons. Notice that it is always possible to go from the case $p_1=0$ to non-vanishing $p_1$ by adding chains of digons to every edge, i.e. performing the change of variables $t\mapsto t/(1-tp_1)$.

\subsubsection{Bipartite quad-hexangulations} For $d=3$ we find a new system. To save some space we set $p_1=0$ (which can always be added back as mentioned above). The equation $E_1(t)=0$ is
\begin{multline}
	8 p_3 t \Fd_1{}'(t) + 4 p_3 \Fd_1(t) = -2 p_3^3 t^{11} \Fd_1{}^{(5)}(t) - \frac{45}{2} p_3^3 t^9 \Fd_1{}''(t)^2 - 30 p_3^3 t^9 \Fd_1{}'(t) \Fd_1{}^{(3)}(t) \\- 40 p_3^3 t^10 \Fd_1{}^{(4)}(t) - 30 p_3^3 t^7 \Fd_1{}'(t)^3 - 270 p_3^3 t^8 \Fd_1{}'(t) \Fd_1{}''(t) - 240 p_3^3 t^9 \Fd_1{}^{(3)}(t) \\- 270 p_3^3 t^7 \Fd_1{}'(t)^2 - 480 p_3^3 t^8 \Fd_1{}''(t) - 240 p_3^3 t^7 \Fd_1{}'(t) - 10 p_2 p_3^2 t^7 \Fd_2{}^{(3)}(t) + 90 p_3^3 t^5 \Fd_2{}'(t)^2 \\- 60 p_2 p_3^2 t^6 \Fd_2{}''(t) + 10 p_3^2 t^6 \Fd_1{}^{(3)}(t) - 60 p_2 p_3^2 t^5 \Fd_2{}'(t) + 30 p_3^2 t^4 \Fd_1{}'(t)^2 + 48 p_3^2 t^5 \Fd_1{}''(t) \\+ 36 p_3^2 t^4 \Fd_1{}'(t) - 3 p_2 p_3 t^3 (u+v) \Fd_1{}'(t) + 6 p_3^2 t^3 \Fd_1(t) \Fd_1{}'(t) + 6 p_2^3 t^3 \Fd_2{}'(t) + 36 p_3^2 t^3 (u+v) \Fd_2{}'(t) \\- 3 p_2^2 t^2 \Fd_1{}'(t) + 22 p_2 p_3 t^2 \Fd_2{}'(t) + 12 p_3 t u v + 3 p_2^2 t \Fd_1(t) + 18 p_2 p_3 t \Fd_2(t)
\end{multline}
and $E_2(t)=0$
\begin{multline}
	6 p_3 t \Fd_2{}'(t) + 6 p_3 \Fd_2(t) = 2 p_2 p_3 t^6 \Fd_1{}^{(3)}(t) + 6 p_3^2 t^6 \Fd_2{}^{(3)}(t) + 12 p_2 p_3 t^4 \Fd_1{}'(t)^2 \\+ 12 p_2 p_3 t^5 \Fd_1{}''(t) + 36 p_3^2 t^4 \Fd_1{}'(t) \Fd_2{}'(t) + 36 p_3^2 t^5 \Fd_2{}''(t) + 12 p_2 p_3 t^4 \Fd_1{}'(t) + 36 p_3^2 t^4 \Fd_2{}'(t) \\+ 9 p_3 t^2 (u+v) \Fd_1{}'(t) - 2 p_2^2 t^2 \Fd_2{}'(t) + p_2 t \Fd_1{}'(t) - p_2 \Fd_1(t)
\end{multline}
Here it is clear how to solve these equations as recurrence relations: for each $n\geq 1$, $E_1(t)=0$ gives a recurrence for $f^{(1)}_n$ in terms of $f^{(1)}_{n'}$ and $f^{(2)}_{n'}$ for $n'<n$, while $E_2(t)=0$ gives $f^{(2)}_n$ in terms of $f^{(1)}_{n'}$ for $n'\leq n$ and $f^{(2)}_{n'}$ for $n'<n$. This is as predicted by Theorem \ref{thm:Main}.

\subsubsection{Bipartite Octangulations} We set $d=4$ and to save some space, also $p_2=p_3=0$. Then $E_1(t)=0$ is
\begin{multline}
	135 t  \Fd_1{}'(t) + 45  \Fd_1(t) = 23 p_4^2 t^{15}  \Fd_1{}^{(7)}(t) + 875 p_4^2 t^{13}  \Fd_1{}^{(3)}(t)^2 \\+ 1190 p_4^2 t^{13}  \Fd_1{}''(t)  \Fd_1{}^{(4)}(t) + 364 p_4^2 t^{13}  \Fd_1{}'(t)  \Fd_1{}^{(5)}(t) + 966 p_4^2 t^{14}  \Fd_1{}^{(6)}(t) \\+ 2100 p_4^2 t^{11}  \Fd_1{}'(t)  \Fd_1{}''(t)^2 + 1400 p_4^2 t^{11}  \Fd_1{}'(t)^2  \Fd_1{}^{(3)}(t) + 24780 p_4^2 t^{12}  \Fd_1{}''(t)  \Fd_1{}^{(3)}(t) \\+ 9660 p_4^2 t^{12}  \Fd_1{}'(t)  \Fd_1{}^{(4)}(t) + 14490 p_4^2 t^{13}  \Fd_1{}^{(5)}(t) + 420 p_4^2 t^9  \Fd_1{}'(t)^4 \\+ 16800 p_4^2 t^{10}  \Fd_1{}'(t)^2  \Fd_1{}''(t) + 74340 p_4^2 t^{11}  \Fd_1{}''(t)^2 + 82740 p_4^2 t^{11}  \Fd_1{}'(t)  \Fd_1{}^{(3)}(t) \\+ 96600 p_4^2 t^{12}  \Fd_1{}^{(4)}(t) + 16800 p_4^2 t^9  \Fd_1{}'(t)^3 + 264600 p_4^2 t^{10}  \Fd_1{}'(t)  \Fd_1{}''(t) + 289800 p_4^2 t^{11}  \Fd_1{}^{(3)}(t) \\- 420 p_4^2 t^{11}  \Fd_3{}^{(5)}(t) + 132300 p_4^2 t^9  \Fd_1{}'(t)^2 + 347760 p_4^2 t^{10}  \Fd_1{}''(t) - 2100 p_4^2 t^9  \Fd_2{}''(t)^2 \\- 2520 p_4^2 t^9  \Fd_2{}'(t)  \Fd_2{}^{(3)}(t) - 1260 p_4^2 t^9  \Fd_1{}^{(3)}(t)  \Fd_3{}'(t) - 2940 p_4^2 t^9  \Fd_1{}''(t)  \Fd_3{}''(t) \\- 2100 p_4^2 t^9  \Fd_1{}'(t)  \Fd_3{}^{(3)}(t) - 8400 p_4^2 t^{10}  \Fd_3{}^{(4)}(t) + 115920 p_4^2 t^9  \Fd_1{}'(t) \\- 3360 p_4^2 t^7  \Fd_1{}'(t)  \Fd_2{}'(t)^2 - 23520 p_4^2 t^8  \Fd_2{}'(t)  \Fd_2{}''(t) - 1680 p_4^2 t^7  \Fd_1{}'(t)^2  \Fd_3{}'(t) \\- 13440 p_4^2 t^8  \Fd_1{}''(t)  \Fd_3{}'(t) - 18480 p_4^2 t^8  \Fd_1{}'(t)  \Fd_3{}''(t) - 50400 p_4^2 t^9  \Fd_3{}^{(3)}(t) - 23520 p_4^2 t^7  \Fd_2{}'(t)^2 \\- 31920 p_4^2 t^7  \Fd_1{}'(t)  \Fd_3{}'(t) - 100800 p_4^2 t^8  \Fd_3{}''(t) - 252 p_4 t^7 (u+v)  \Fd_1{}^{(3)}(t) - 50400 p_4^2 t^7  \Fd_3{}'(t) \\- 336 p_4 t^5 (u+v)  \Fd_1{}'(t)^2 - 1512 p_4 t^6 (u+v)  \Fd_1{}''(t) + 2520 p_4^2 t^5  \Fd_3{}'(t)^2 - 1512 p_4 t^5 (u+v)  \Fd_1{}'(t)  \\+ 504 p_4 t^6  \Fd_2{}^{(3)}(t) + 1512 p_4 t^4  \Fd_1{}'(t)  \Fd_2{}'(t) + 2472 p_4 t^5  \Fd_2{}''(t) + 288 p_4 t^3  \Fd_2(t)  \Fd_1{}'(t) \\+ 1920 p_4 t^4  \Fd_2{}'(t) + 120 p_4 t^3  \Fd_1(t)  \Fd_2{}'(t) + 720 p_4 t^3 (u+v)  \Fd_3{}'(t) + 180 t u v
\end{multline}
and $E_2(t)=0$,
\begin{multline}
	18 t  \Fd_2{}'(t) + 12  \Fd_2(t) = -p_4 t^{10}  \Fd_1{}^{(5)}(t) - 15 p_4 t^8  \Fd_1{}''(t)^2 - 10 p_4 t^8  \Fd_1{}'(t)  \Fd_1{}^{(3)}(t) \\- 20 p_4 t^9  \Fd_1{}^{(4)}(t) - 120 p_4 t^7  \Fd_1{}'(t)  \Fd_1{}''(t) - 120 p_4 t^8  \Fd_1{}^{(3)}(t) - 120 p_4 t^6  \Fd_1{}'(t)^2 \\- 240 p_4 t^7  \Fd_1{}''(t) - 120 p_4 t^6  \Fd_1{}'(t) + 30 p_4 t^6  \Fd_3{}^{(3)}(t) + 60 p_4 t^4  \Fd_2{}'(t)^2 \\+ 120 p_4 t^4  \Fd_1{}'(t)  \Fd_3{}'(t) + 180 p_4 t^5  \Fd_3{}''(t) + 180 p_4 t^4  \Fd_3{}'(t) + 24 t^2 (u+v)  \Fd_1{}'(t)
\end{multline}
and $E_3(t)=0$,
\begin{multline}
	6 t  \Fd_3{}'(t) + 6  \Fd_3(t) = -2 p_4 t^{10}  \Fd_2{}^{(5)}(t) - 12 p_4 t^8  \Fd_1{}^{(3)}(t)  \Fd_2{}'(t) - 24 p_4 t^8  \Fd_1{}''(t)  \Fd_2{}''(t) \\- 16 p_4 t^8  \Fd_1{}'(t)  \Fd_2{}^{(3)}(t) - 40 p_4 t^9  \Fd_2{}^{(4)}(t) - 24 p_4 t^6  \Fd_1{}'(t)^2  \Fd_2{}'(t) - 120 p_4 t^7  \Fd_1{}''(t)  \Fd_2{}'(t) \\- 144 p_4 t^7  \Fd_1{}'(t)  \Fd_2{}''(t) - 240 p_4 t^8  \Fd_2{}^{(3)}(t) - 264 p_4 t^6  \Fd_1{}'(t)  \Fd_2{}'(t) - 480 p_4 t^7  \Fd_2{}''(t) \\- 240 p_4 t^6  \Fd_2{}'(t) + 48 p_4 t^4  \Fd_2{}'(t)  \Fd_3{}'(t) + 3 t^5  \Fd_1{}^{(3)}(t) + 6 t^3  \Fd_1{}'(t)^2 + 13 t^4  \Fd_1{}''(t) \\+ 8 t^3  \Fd_1{}'(t) + 2 t^2  \Fd_1(t)  \Fd_1{}'(t) + 8 t^2 (u+v)  \Fd_2{}'(t).
\end{multline}
Just like the case $d=2$, it is easy to see how to implement recurrence relations from those three ODEs: for each $n\geq 1$, $E_1(t)$ gives $f^{(1)}_n$ in terms of $f^{(1)}_{n'}$, $f^{(2)}_{n'}$ and $f^{(3)}_{n'}$ for $n'<n$, then $E_2(t)=0$ gives $f^{(2)}_n$ in terms of $f^{(1)}_{n'}$ for $n'\leq n$ and $f^{(2)}_{n'}$ and $f^{(3)}_{n'}$ for $n'<n$, then $E_3(t)=0$ gives $f^{(3)}_n$ in terms of $f^{(1)}_{n'}$ and $f^{(2)}_{n'}$ for $n'\leq n$, and $f^{(3)}_{n'}$ for $n'<n$. This is again as predicted by Theorem \ref{thm:Main}. Below we prove that the system of ODEs we establish always yields recurrence relations of this form.

\section{Recurrence relations} \label{sec:Recurrence}
We are interested in evaluating inductively the number of rooted maps, i.e. the coefficients of $t\Fd{}'(t) = \sum_{n\geq 1} f_n t^n$ where $f_n\equiv f_n(u,v,\pi_d(\mathbf{p}))\in \QQ[u,v,\pi_d(\mathbf{p})]$. For $k=1, \dotsc, d-1$, let
\begin{equation} \label{Series}
	\Fd_k(t) = \sum_{n\geq 1} f^{(k)}_n(u,v,\pi_d(\bfp)) t^{n},\quad \text{with $f^{(k)}_n\in\QQ[u,v,\pi_d(\mathbf{p})]$},
\end{equation}
Notice that $f^{(k)}_{n} = 0$ for $1\leq n<k$ since there are no maps with a root face of degree $k$ if the total size is less than $k$. The first Virasoro constraint $F_1 = t^2 \dt F + tuv$ gives the relation $f^{(1)}_n = f_{n-1}$ for $n\geq 2$ and $f^{(1)}_1=uv$. It is thus equivalent to calculate $f_n$, or $f^{(1)}_n$, which we do. The rest of this section is devoted to proving the following theorem.

\begin{thm} \label{thm:Recurrence}
	 The equations from Theorem \ref{thm:ODEs} give rise to recurrence formulas which determine recursively $f^{(1)}_n, \dotsc, f^{(d-1)}_n$ for $n\geq 1$.
\end{thm}

First, we make sure that the appropriate coefficients of $f^{(1)}_n, \dotsc, f^{(d-1)}_n$ do not vanish, with the following lemma.
\begin{lemma}
	The coefficient of $f^{(k)}_n$ in $[t^{n-\delta_{k,1}}]E_k(t)$ does not vanish for any $n\geq 1$ for all $k=1, \dotsc, d-1$.
\end{lemma}

The shift of index by $\delta_{k,1}$ is due to the double substitution in $G_{2d-1}$.

\begin{proof}
	Notice that the coefficient of $f^{(k)}_n$ in $[t^{n-\delta_{k,1}}]E_k(t)$ is a polynomial in $u, v$, $p_d^{-1}$, $\pi_d(\bfp)$. Therefore it is enough to prove that it is not identically zero in the special case $p_1=p_2 = \dotsb = p_{d-1}=0$. We will work with this reduction in the proof, but the beginning of the discussion is independent of this reduction.
	
	Let $k\in\{2, \dotsc, d-1\}$. By definition, the coefficient of $f^{(k)}_n$ in $[t^{n}]E_k(t)$ solely comes from terms that are linear in $\Fd_{k}(t)$ and its $t$-derivatives and that are moreover of the form $Q(u,v,\pi_d(\mathbf{p}), p_d^{-1})t^{r} \dt^r\Fd_k(t)$ for a polynomial $Q$, and all possible $r\geq 0$. As for the coefficient of $f^{(1)}_n$ in $[t^{n-1}]E_1(t)$, it comes from terms of the form $\tilde{Q}(u,v,\pi_d(\mathbf{p}), p_d^{-1})t^{r-1} \dt^{r}(\Fd_1(t))$. Tracing their origins before the operator $T$ is applied, these are terms respectively of the form $Q(u,v,\pi_d(\mathbf{p}), p_d^{-1})t^{-r} \Fd_{k,1^r}$ and $\tilde{Q}(u,v,\pi_d(\mathbf{p}), p_d^{-1})t^{-r-1} \Fd_{1^{r+1}}$. Therefore we look for the terms 
	\begin{equation*}
		\text{$[t^{-r} \Fd_{k,1^r}] \mathcal{E}_k(\mathcal{G}^{\pi_d})$ and $[t^{-r-1} \Fd_{1^{r+1}}]\mathcal{E}_1(\mathcal{G}^{\pi_d})$ for all $r\geq 0$.}
	\end{equation*}
	
	In $E_k(\mathcal{G}^{\pi_d})$, there is $\Fd_{k,1}/t - k\Fd_k$ (respectively contributing with $r=1$ and $r=0$). We now concentrate on $\sum_{i=1}^{d} ip_i \operatorname{KP}_{i,k}(\mathcal{G}^{\pi_d})$. Let $i\in\{1, \dotsc, d\}$. In $\operatorname{KP}_{i,k}(\mathcal{G}^{\pi_d})$, the non-linear terms in $\mathcal{G}^{\pi_d}$ cannot give rise to the linear terms we are looking for. The linear terms in $\operatorname{KP}_{i,k}(\mathcal{G}^{\pi_d})$ are of the form $\sum_{s=1}^{i+k-1} \alpha^{(i,k)}_s G_{i+k-s,1^s}$ for some $\alpha^{(i,k)}_s\in\mathbb{Q}$. If $i+k-s<d$, then $G_{i+k-s,1^s} = \Fd_{i+k-s,1^s}$, so the terms we want, having $F_{k,1^r}$, would correspond to $s=i=r$. However, the terms we want also have a factor $t^{-r}$ which is not the case unless $r=0$, contradicting $r=s\geq 1$.
	
	It comes that $i+k-s\geq d$ with $s\geq 1$ and thus
	\begin{equation*}
		-\sum_{i=1}^d ip_i [t^{-r} \Fd_{k,1^r}] \operatorname{KP}_{i,k}(\mathcal{G}^{\pi_d}) = -\sum_{i=1}^d ip_i \sum_{s=1}^{i+k-d} \alpha^{(i,k)}_s [t^{-r} \Fd_{k,1^r}] G_{i+k-s,1^s}(\mathcal{F}_{\text{$d$-adm.}}).
	\end{equation*}
	In this case, $G_{i+k-s,1^s}$ is given by the recurrence \eqref{G}. To simplify the rest of the discussion we specialize the variables $\mathbf{p}=(p_1, p_2, \dotsc)$ as announced using $\bar{\pi}_d (p_i) = \delta_{i,d} p_d$ so that
	\begin{equation*}
		-\bar{\pi}_d\sum_{i=1}^d ip_i [t^{-r} \Fd_{k,1^r}] \operatorname{KP}_{i,k}(\mathcal{G}^{\pi_d}) = -dp_d \sum_{s=1}^{k} \alpha^{(d,k)}_s [t^{-r} F^{\bar{\pi}_d}_{k,1^r}] G_{n+k-s,1^s}(\bar{\pi}_d\mathcal{F}_{\text{$d$-adm.}}),
	\end{equation*}
	with $F^{\bar{\pi}_d}_\lambda \coloneqq \bar{\pi}_d \Fd_\lambda$. The recurrence \eqref{G} becomes for $m=0, \dotsc, d-2$
	\begin{multline} \label{barG}
		(d+m)tp_dG_{d+m,1^l}(\bar{\pi}_d \mathcal{F}_{\text{$d$-adm.}}) =	- t\sum_{\substack{i,j\geq 1\\ i+j=m}} ij\Bigl(\sum_{\substack{a, b\geq0\\ a+b=l}} \binom{l}{r} F^{\bar{\pi}_d}_{i,1^a} F^{\bar{\pi}_d}_{j,1^b} + \operatorname{KP}_{i,j,1^l}(\bar{\pi}_d\mathcal{F}_{\text{hook}})\Bigr)\\
		+ (m+1)F^{\bar{\pi}_d}_{m+1,1^l} - t(u+v)mF^{\bar{\pi}_d}_{m,1^l} - tl(m+1) F^{\bar{\pi}_d}_{m+1,1^{l-1}} - tuv\delta_{m,0}\delta_{l,0},
	\end{multline}
	and for $m=d-1$
	\begin{multline} \label{barG2d-1}
		(2d-1)tp_dG_{2d-1,1^l}(\bar{\pi}_d \mathcal{F}_{\text{$d$-adm.}}) = - t\sum_{\substack{i,j\geq 1\\ i+j=d-1}} ij\Bigl(\sum_{\substack{a, b\geq0\\ a+b=l}} \binom{l}{r} F^{\bar{\pi}_d}_{i,1^a} F^{\bar{\pi}_d}_{j,1^b} + \operatorname{KP}_{i,j,1^l}(\bar{\pi}_d\mathcal{F}_{\text{hook}})\Bigr)\\
		+ d G_{d,1^l}(\bar{\pi}_d\mathcal{F}_{\text{$d$-adm.}}) - t(u+v)(d-1)F^{\bar{\pi}_d}_{d-1,1^l} - tld G_{d,1^{l-1}}(\bar{\pi}_d\mathcal{F}_{\text{$d$-adm.}}),
	\end{multline}
	with $G_{d,1^l}(\bar{\pi}_d\mathcal{F}_{\text{$d$-adm.}}) = (F^{\bar{\pi}_d}_{1^{l+1}} - tl F^{\bar{\pi}_d}_{1^l} - tuv\delta_{l,0})/(tdp_d)$.
	
	We find that the only occurences of $t^{-r} F^{\bar{\pi}_d}_{k,1^r}$ in \eqref{barG} and \eqref{barG2d-1} are in \eqref{barG} for $m=k-1$, with contributions $-\frac{k}{(d+k-1)p_d} F^{\bar{\pi}_d}_k$ for $r=0$ and $\frac{k}{(d+k-1)p_d} t^{-1}F^{\bar{\pi}_d}_{k,1}$ for $r=1$. Hence
	\begin{equation*}
		\begin{aligned}
			-dp_d \sum_{s=1}^{k} \alpha^{(d,k)}_s [t^{-r} F^{\bar{\pi}_d}_{k,1^r}] G_{d+k-s,1^s}(\mathcal{F}_{\text{$d$-adm.}}) &= -dp_d \alpha^{(d,k)}_1 [t^{-r} F^{\bar{\pi}_d}_{k,1^r}] G_{d+k-1,1}(\bar{\pi}_d\mathcal{F}_{\text{$d$-adm.}}),\\
			&= \frac{dk}{d+k-1} (\delta_{r,0} - \delta_{r,1}),
		\end{aligned}
	\end{equation*}
	using $\alpha^{(d,k)}_1=1$.	Overall we find the coefficient of $f^{(k)}_n$ in $[t^{n}]\bar{\pi}_d E_k(t)$ by computing
	\begin{equation*}
		\begin{aligned}
			J_k&\equiv \sum_{r} t^{-r} F^{\bar{\pi}_d}_{k,1^r} [t^{-r} F^{\bar{\pi}_d}_{k,1^r}] \bar{\pi}_d E_k(\mathcal{G}^{\pi_d}) \\
			&= t^{-1}F^{\bar{\pi}_d}_{k,1} - kF^{\bar{\pi}_d}_k  + \frac{dk}{(d+k-1)} F^{\bar{\pi}_d}_k  - \frac{dk}{(d+k-1)} t^{-1}F^{\bar{\pi}_d}_{k,1}\\
			&= \frac{-k(k-1)}{d+k-1}F^{\bar{\pi}_d}_k - \frac{(d-1)(k-1)}{d+k-1} t^{-1} F^{\bar{\pi}_d}_{k,1},
		\end{aligned}
	\end{equation*}
	hence 
	\begin{equation*}
			TJ_k = -\frac{k-1}{d+k-1}\sum_{n\geq 1} \bigl(n(d-1) + k\bigr) f^{(k)}_n t^{n}.
	\end{equation*}
	We conclude by observing that $-\frac{k-1}{d+k-1}(n(d-1) + k)$ does not vanish for any $n\geq 1$ and $k\geq 2$.
	
	We now treat $[t^{-r-1} \Fd_{1^{r+1}}]E_1(\mathcal{G}^{\pi_d})$. It gets contributions at $r=0$ and $r=1$ from $\frac{\Fd_{1,1}-t(d+1)\Fd_1}{t^2dp_d}$. We turn our attention to $[t^{-r-1} \Fd_{1^{r+1}}]\Bigl(-\sum_{i=1}^{d} ip_i\operatorname{KP}_{i,d}(\mathcal{G}^{\pi_d})\Bigr)$. With the same arguments as we used for $E_k(\mathcal{G}^{\pi_d})$, we can focus on the linear terms from $\operatorname{KP}_{i,d}(\mathcal{G}^{\pi_d})$, and by applying $\bar{\pi}_d$ straight away we get
	\begin{equation*}
		-\bar{\pi}_d [t^{-r-1} \Fd_{1^{r+1}}]\sum_{i=1}^d ip_i \operatorname{KP}_{i,d}(\mathcal{G}^{\pi_d}) = -dp_d [t^{-r-1} F^{\bar{\pi}_d}_{1^{r+1}}] \sum_{s=1}^{2d-1} \alpha^{(d,d)}_s G_{2d-s,1^s}(\bar{\pi}_d\mathcal{F}_{\text{$d$-adm.}}),
	\end{equation*}
	for $\alpha^{(d,d)}_s\in\mathbb{Q}$. Also similarly to the case of $E_k(\mathcal{G}^{\pi_d})$, only the terms $2d-s \geq d$, i.e. $d\geq s$, may contribute.
	
	Thus, $G_{2d-s,1^s}(\bar{\pi}_d\mathcal{F}_{\text{$d$-adm.}})$ is given by the recurrences \eqref{barG} and \eqref{barG2d-1}. Then $t^{-r-1}F^{\bar{\pi}_d}_{1^{r+1}}$ only appears for $s=1$, and
	\begin{equation*}
		\begin{aligned}
			-dp_d[t^{-r-1} F^{\bar{\pi}_d}_{1^{r+1}}] G_{2d-1,1}(\bar{\pi}_d \mathcal{F}_{\text{$d$-adm.}}) &= \frac{-d}{(2d-1)p_d} [t^{-r-1} F^{\bar{\pi}_d}_{1^{r+1}}] \Bigl(\frac{F^{\bar{\pi}_d}_{1^2} - 2tF^{\bar{\pi}_d}_1}{t^2} \Bigr)\\
			&= \frac{-d}{(2d-1)p_d} \delta_{r,1} + \frac{2d}{(2d-1)p_d}\delta_{r,0}.
		\end{aligned}		
	\end{equation*}
	Therefore we find the coefficients of $f^{(1)}_n$ in $[t^{n-1}]\bar{\pi}_d E_1(t)$ by computing
	\begin{equation*}
		\begin{aligned}
			J_1&\equiv \sum_{r} t^{-r-1} F^{\bar{\pi}_d}_{1^{r+1}} [t^{-r-1} F^{\bar{\pi}_d}_{1^{r+1}}] \bar{\pi}_d E_1(\mathcal{G}^{\pi_d})
			= \frac{F^{\bar{\pi}_d}_{1,1}-t(d+1)F^{\bar{\pi}_d}_1}{t^2dp_d} - d\frac{F^{\bar{\pi}_d}_{1,1}-2tF^{\bar{\pi}_d}_1}{(2d-1)p_d} \\
			&= -\frac{d-1}{d(2d-1)p_d}\bigl((d-1) t^{-2}F^{\bar{\pi}_d}_{1^2} + t^{-1} F^{\bar{\pi}_d}_1\bigr).
		\end{aligned}
	\end{equation*}
	Substituting the series in $t$ we get
	\begin{equation*}
		\begin{aligned}
			dp_dTJ_1 &= -\frac{d-1}{2d-1}\bigl((d-1)F^{\bar{\pi}_d}{}'(t) + t^{-1} F^{\bar{\pi}_d}_1(t)\bigr),\\
			&= -\frac{d-1}{2d-1}\sum_{n\geq 1} \bigl((d-1)n + 1\bigr) f^{(1)}_n t^{n-1}.
		\end{aligned}
	\end{equation*}
	We conclude our proof by noticing that $(d-1)n + 1>0$ for any $n\geq 1$.
\end{proof}

The following lemma extracts coefficients with respect to $t$ from respectively linear and non-linear terms in $\mathcal{F}_{\text{$d$-adm.}}$.

\begin{lemma} \label{lemma:Extraction}
	\begin{itemize}
		\item Let $e\in\ZZ$, $i, n\geq 1$ and $l\geq 0$. Then there is a polynomial $\chi^{(i,l)}(n')$ in $n'$ such that
		\begin{equation*}
			[t^{n}] t^e T\Fd_{i,1^l} = f^{(i)}_{n-(l+e)} \chi^{(i,l)}(n-(l+e)).
		\end{equation*}
		\item Let $a\geq 2$ and $(i_1, \dotsc, i_a) \in\{1, \dotsc, d-1\}^a$ and $l_1, \dotsc, l_a \geq 0$. Then for all $b=1, \dotsc, a$, the coefficient $[t^{n}] t^e T \prod_{\alpha=1}^a \Fd_{i_\alpha, 1^{l_\alpha}}$ involves $f^{(i_b)}_{n_b}$ with $n_b<n-(l_b+e)$ only.
	\end{itemize}
\end{lemma}

\begin{proof}
	Note that $T\Fd_{i,1^l} = (t^2\dt)^l \Fd_i(t)$ hence there is a polynomial $\chi^{(i,l)}(n')$ such that $T\Fd_{i,1^l} = \sum_{n'\geq 1} \chi^{(i,l)}(n') f^{(i)}_{n'} t^{n'+l}$. This proves the first bullet point.
	
	For the second one, we have with the same polynomial $\chi^{(i,l)}(n')$
	\begin{equation*}
		t^e T \prod_{\alpha=1}^a \Fd_{i_\alpha, 1^{l_\alpha}} = \sum_{n_1, \dotsc, n_a\geq 1} t^{e + \sum_{\alpha=1}^a(n_\alpha+l_\alpha)} \prod_{\alpha=1}^a \chi^{(i_\alpha, l_\alpha)}(n_\alpha) f^{(i_\alpha)}_{n_\alpha}.
	\end{equation*}
	Then extracting the coefficients of $t^{n}$ gives the equation $n = e + \sum_{\alpha=1}^a(n_\alpha+l_\alpha)$ and the result follows by noting that $n_\alpha>0$ for all $\alpha$.
\end{proof}

\begin{proof}[Proof of Theorem \ref{thm:Recurrence}]
	Let $k\in\{1, \dotsc, d-1\}$ and $n\geq 1$. The goal is here to show that $[t^{n-\delta_{k,1}}] E_k(t)$ only involves
	\begin{enumerate}[label=(\arabic*)]
		\item\label{enum:1} coefficients $f^{(i)}_{n'}$ with $n'\leq n$ if $i\leq k$,
		\item\label{enum:2} coefficients $f^{(i)}_{n'}$ with $n'<n$ if $i>k$.
	\end{enumerate}
	This is enough to transform the ODEs into the recurrence relations of Theorem \ref{thm:Main}. We call them properties \ref{enum:1} and \ref{enum:2}.
	
	We consider first the terms that are linear in $f^{(i)}_{n'}$. Rewinding the construction of the differential equation $E_k(t)=0$, these linear terms come from applying $T$ to $\mathcal{E}_k(\mathcal{G}^{\pi_d})$ and selecting the terms that are linear in $\mathcal{F}_{\text{$d$-adm.}}$. For notational purposes we introduce the operator $L = \sum_{i=1}^{d-1} \sum_{l\geq 0}\Fd_{i,1^l} [\Fd_{i,1^l}]$ which is the projection onto the linear terms in $\QQ[t,t^{-1},u,v,p_d^{-1},\pi_d(\mathbf{p})][\mathcal{F}_{\text{$d$-adm.}}]$. 
	
	Let $i, k\geq 2$, then the linear terms of $\operatorname{KP}_{i,k}(\mathcal{F})$ are of the general form $F_{s, 1^{i+k-s}}$ for $s=1, \dotsc, i+k-1$ since they expand over hooks with at least one $p_1$-derivative. Then for all $k=2, \dotsc, d-1$
	\begin{equation*}
		L\mathcal{E}_k(\mathcal{G}^{\pi_d})\equiv \frac{1}{t}\Fd_{k,1} - k\Fd_k - \sum_{i=1}^d ip_i \sum_{s=1}^{i+k-1} \kappa_{i,k,s} LG_{s,1^{i+k-s}}(\mathcal{F}_{\text{$d$-adm.}})
	\end{equation*}
	where $\kappa_{i,k,s}\in\QQ$. In terms of $d$-admissible quantities, the last sum is
	\begin{multline}
		\sum_{i=1}^d ip_i \sum_{s=1}^{i+k-1} \kappa_{i,k,s} LG_{s,1^{i+k-s}}(\mathcal{F}_{\text{$d$-adm.}}) = \sum_{i=1}^d ip_i \sum_{s=1}^{d-1} \kappa_{i,k,s} \Fd_{s,1^{i+k-s}} \\+ \sum_{i=d+1-k}^d ip_i \sum_{s=d}^{i+k-1} \kappa_{i,k,s} LG_{s,1^{i+k-s}}(\mathcal{F}_{\text{$d$-adm.}}).
	\end{multline}
	After a change of indices, the last sum is 
	\begin{equation*}
		\sum_{i=d+1-k}^d ip_i \sum_{s=d}^{i+k-1} \kappa_{i,k,s} LG_{s,1^{i+k-s}}(\mathcal{F}_{\text{$d$-adm.}}) = \sum_{m=0}^{k-1} \sum_{l=1}^{k-m} \tilde{\kappa}_{l,k,m} LG_{d+m, 1^l}(\mathcal{F}_{\text{$d$-adm.}})
	\end{equation*}
	for some $\tilde{\kappa}_{l,k,m}$ that are polynomials in $\pi_d(\mathbf{p})$. We can then substitute $G_{d+m, 1^l}(\mathcal{F}_{\text{$d$-adm.}})$ from the above sum with his $d$-admissible expression \eqref{Gorder}
	\begin{multline} \label{LinG}
		LG_{d+m,1^l}(\mathcal{F}_{\text{$d$-adm.}}) = \sum_{i=1}^{m+1} \Bigl(\alpha^{(m,l)}_{i} t^{-1}\Fd_{i,1^l} + \beta^{(m,l)}_{i} \Fd_{i,1^{l-1}}\Bigr) + \sum_{i=1}^{d-1} \gamma^{(m,l)}_i \Fd_{i,1^l} \\
		+ \mathbf{1}_{m\geq 2}\sum_{i=1}^m \sum_{r=1}^{i-1} \rho^{(m,l)}_{i,r} \Fd_{i-r,1^{l+r}}.
	\end{multline}	
	The next step is to apply the operator $T$ to the resulting expression of $L\mathcal{E}_k(\mathcal{G}^{\pi_d})$,
	\begin{multline} \label{LinearEk}
		TL\mathcal{E}_k(\mathcal{G}^{\pi_d}) = t^{-1} T\Fd_{k,1} - kT\Fd_k - \sum_{i=1}^d ip_i \sum_{s=1}^{d-1} \kappa_{i,k,s} T\Fd_{s,1^{i+k-s}}\\
		- \sum_{m=0}^{k-1} \sum_{l=1}^{k-m} \tilde{\kappa}_{l,k,m} \sum_{i=1}^{m+1} \Bigl(\alpha^{(m,l)}_{i} t^{-1}T\Fd_{i,1^l} + \beta^{(m,l)}_{i} T\Fd_{i,1^{l-1}}\Bigr) - \sum_{m=0}^{k-1} \sum_{l=1}^{k-m} \tilde{\kappa}_{l,k,m} \sum_{i=1}^{d-1} \gamma^{(m,l)}_i T\Fd_{i,1^l}\\
		- \sum_{m=2}^{k-1} \sum_{l=1}^{k-m} \tilde{\kappa}_{l,k,m} \sum_{i=1}^m \sum_{r=1}^{i-1} \rho^{(m,l)}_{i,r} T\Fd_{i-r,1^{l+r}},
	\end{multline}
	and extract the coefficient of $[t^{n}]$ using Lemma \ref{lemma:Extraction}. We separate the cases as follows.
	\begin{itemize}
		\item All the terms that are of the form $t^e T\Fd_{i,1^l}$ with $e=0$ and $l\geq1$ contribute to $[t^{n}]E_k(t)$ with $f^{(i)}_{n'}$ with $n'=n-l<n$, times some polynomial evaluated at $n'$. For example,
		\begin{equation*}
			\sum_{m=0}^{k-1} \sum_{l=1}^{k-m} \tilde{\kappa}_{l,k,m} \sum_{i=1}^{d-1} \gamma^{(m,l)}_i T\Fd_{i,1^l} = \sum_{n'\geq 1} \sum_{i=1}^{d-1} \sum_{l=1}^k \tilde{\gamma}^{(k)}_{i,l}(n') f^{(i)}_{n'} t^{l+n'},
		\end{equation*}
		for some $\tilde{\gamma}^{(k)}_{i,l}(n')$ polynomial in $n'$. Properties \ref{enum:1} and \ref{enum:2} are thus satisfied for those terms. This accounts for all terms on \eqref{LinearEk} except the following ones.
		\item $t^{-1} T\Fd_{k,1} - kT\Fd_k$ contributes to $[t^{n}]E_k(t)$ with $f^{(k)}_{n}$ times a polynomial in $n$, thus satisfying property \ref{enum:1}.
		\item $\alpha^{(m,l)}_{i} t^{-1}T\Fd_{i,1^l}$ contributes with $f^{(i)}_{n'}$ with $n'=n-l-1\leq n$, times a polynomial in $n'$, and $i$ ranges from 1 to $k$,
		\begin{equation*}
			\sum_{m=0}^{k-1} \sum_{l=1}^{k-m} \tilde{\kappa}_{l,k,m} \sum_{i=1}^{m+1} \alpha^{(m,l)}_{i} t^{-1}T\Fd_{i,1^l} = \sum_{n'\geq 1} \sum_{i=1}^k \sum_{l=1}^k \tilde{\alpha}^{(k)}_{i,l}(n') f^{(i)}_{n'} t^{l-1+n'},
		\end{equation*}
		for some $\tilde{\alpha}^{(k)}_{i,l}(n')$ polynomial in $n'$. This gives the weak inequality of property \ref{enum:1}.
		\item $\beta^{(m,l)}_{i} T\Fd_{i,1^{l-1}}$ also contributes with $f^{(i)}_{n'}$ with $n'=n-l-1\leq n$, times a polynomial in $n'$, while $i$ also ranges from 1 to $k$.
	\end{itemize}
	This proves properties \ref{enum:1} and \ref{enum:2}.

	Consider now $L\mathcal{E}_1(\mathcal{G}^{\pi_d})$, which requires calculating $L\operatorname{KP}_{i,d}(\mathcal{G}^{\pi_d})$ for $i=1, \dotsc, d$. It expands in the form $G_{s, 1^{i+d-s}}(\mathcal{F}_{\text{$d$-adm.}})$ with $s=1, \dotsc, i+d-1$, hence
	\begin{multline}
		L\mathcal{E}_1(\mathcal{G}^{\pi_d}) =  -\sum_{i=1}^{d} ip_i \sum_{s=1}^{i+d-1} \kappa_{i,s} LG_{s,1^{i+d-s}}(\mathcal{F}_{\text{$d$-adm.}}) + \frac{\Fd_{1,1}-t(d+1)\Fd_1}{t^2dp_d} \\
		- \sum_{i=1}^{d-1} \frac{ip_i}{dp_d} \Bigl(\frac{\Fd_{i,1}}{t} - d\Fd_i\Bigr).
	\end{multline}
	Following the same steps as for $k\geq 2$,
	\begin{multline} \label{LinearE1}
		L\mathcal{E}_1(\mathcal{G}^{\pi_d}) =  -\sum_{i=1}^d ip_i \sum_{s=1}^{d-1} \kappa_{i,s} \Fd_{s,1^{i+d-s}} - \sum_{m=0}^{d-1} \sum_{l=1}^{d-m} \tilde{\kappa}_{l,m} LG_{d+m, 1^l}(\mathcal{F}_{\text{$d$-adm.}}) \\
		+ \frac{\Fd_{1,1}-t(d+1)\Fd_1}{t^2dp_d} - \sum_{i=1}^{d-1} \frac{ip_i}{dp_d} \Bigl(\frac{\Fd_{i,1}}{t} - d\Fd_i\Bigr),
	\end{multline}
	for some $\tilde{\kappa}_{l,m}$ that are polynomials in $\pi_d(\mathbf{p})$. Then $LG_{d+m, 1^l}(\mathcal{F}_{\text{$d$-adm.}})$ is given precisely by \eqref{LinG} for all $m=0, \dotsc, d-2$. However $m=d-1$ now appears in the sum, in which case one has to consider the $d$-admissible expression \eqref{G2d-1} for $G_{2d-1, 1^l}(\mathcal{F}_{\text{$d$-adm.}})$,
	\begin{multline}
		LG_{2d-1,1^l}(\mathcal{F}_{\text{$d$-adm.}}) = \sum_{i=1}^{d-1} \Bigl(\alpha^{(d-1,l)}_{i} t^{-1}\Fd_{i,1^l} + \beta^{(d-1,l)}_i \Fd_{i,1^{l-1}} + \gamma^{(d-1,l)}_i \Fd_{i,1^l}\Bigr) \\
		+ \sum_{i=1}^{d-1} \sum_{r=1}^{i-1} \rho^{(d-1,l)}_{i,r} \Fd_{i-r,1^{l+r}} + \frac{\Fd_{1^{l+1}}}{(2d-1)p_d^2t^2} - \frac{2l\Fd_{1^l}}{(2d-1)p_d^2t}.
	\end{multline}
	We now apply $T$ and extract the coefficient of $t^n$ using Lemma \ref{lemma:Extraction}. Let us separate various pieces as follows.
	\begin{itemize}
		\item All the terms that are of the form $t^e T\Fd_{i,1^l}$ with $e=0$ and $l\geq1$ contribute to $[t^{n-1}]E_1(t)$ with $f^{(i)}_{n'}$ with $n'=n-l<n$, times some polynomial evaluated at $n'$. This accounts for all terms on \eqref{LinearEk} except the following ones.
		\item $\frac{T\Fd_{1,1}-t(d+1)T\Fd_1}{t^2dp_d}$ contributes to $[t^{n-1}]E_1(t)$ with $f^{(1)}_{n}$ times a polynomial in $n$.
		\item From $LG_{2d-1,1^l}(\mathcal{F}_{\text{$d$-adm.}})$ with $l\geq 1$, we have $\frac{T\Fd_{1^{l+1}}}{(2d-1)p_d^2t^2} - \frac{2lT\Fd_{1^l}}{(2d-1)p_d^2t}$ which contributes with $f^{(1)}_{n'}$ with $n'=n-l+1\leq n$, times a polynomial in $n'$, hence (only $l=1$ appears in the sum for $m=d-1$)
		\begin{equation*}
			\tilde{\kappa}_{1,d-1} \Bigl(\frac{\Fd_{1^{2}}}{(2d-1)p_d^2t^2} - \frac{2\Fd_{1}}{(2d-1)p_d^2t}\Bigr) = \sum_{n'\geq1} \omega(n') f^{(1)}_{n'} t^{n'-1},
		\end{equation*}
		for some polynomials $\omega(n')$.
		\item In \eqref{LinearE1}, we have $- \sum_{i=1}^{d-1} \frac{ip_i}{dp_d} \Bigl(\frac{T\Fd_{i,1}}{t} - dT\Fd_i\Bigr)$, which involves exactly $f^{(i)}_{n-1}$, times some polynomials in $n$.
		\item In $LG_{2d-1,1^l}$ we find also $\alpha^{(d-1,l)}_{i} t^{-1}T\Fd_{i,1^l} + \beta^{(d-1,l)}_i T\Fd_{i,1^{l-1}}$ which gives rise to $f^{(i)}_{n'}$ with $n'=n-l< n$, times some polynomials in $n'$.
	\end{itemize}
	This proves properties \ref{enum:1} and \ref{enum:2} for $TL\mathcal{E}_1$.
	
	We will not give as much detail for the non-linear terms of $E_k(t)$. It is in fact enough to show that in $\mathcal{E}_k(\mathcal{G}^{\pi_d})$, all monomials of the form $t^e \prod_{\alpha=1}^a \Fd_{i_\alpha, 1^{l_\alpha}}$ for $a\geq 2$ (i.e. the non-linear terms in $\mathcal{F}_{\text{$d$-adm.}}$) all have $e=0$ and $l\geq 1$. As a consequence, all $f^{(i)}_{n'}$s appearing in the non-linear terms of $[t^{n-\delta_{k,1}}]E_k(t)$ have $n'<n$ for all $i=1, \dotsc, d-1$, hence satisfy properties \ref{enum:1} and \ref{enum:2}.
\end{proof}

\end{document}